\documentclass[12pt]{amsart}
\usepackage{amscd,amssymb,amsthm,amsmath,amssymb,textcomp,supertabular,rotating,longtable,enumerate,rotating,mathrsfs,mathtools,multirow, url}
\usepackage[matrix,arrow,curve]{xy}
\usepackage{pdflscape}

\usepackage[shortlabels]{enumitem}

\usepackage[margin=2.7cm]{geometry}

\newtheorem{theorem}[equation]{Theorem}
\newtheorem*{theorem*}{Theorem}

\newtheorem{proposition}[equation]{Proposition}
\newtheorem*{proposition*}{Proposition}
\newtheorem{lemma}[equation]{Lemma}
\newtheorem{corollary}[equation]{Corollary}
\newtheorem*{corollary*}{Corollary}

\newtheorem*{problem*}{Problem}

\newtheorem*{question*}{Question}

\newtheorem*{construction*}{Construction}
\newtheorem*{maintheorem*}{Main Theorem}

\theoremstyle{definition}

\newtheorem*{example*}{Example}

\newtheorem*{definition*}{Definition}

\theoremstyle{remark}
\newtheorem{remark}[equation]{Remark}
\newtheorem*{remark*}{Remark}

\makeatletter\@addtoreset{equation}{section} \makeatother

\newcommand{\sslash}{\mathbin{/\mkern-6mu/}}

\author{Ivan Cheltsov, Maksym Fedorchuk, Kento Fujita, Anne-Sophie Kaloghiros}

\title{K-moduli of pure states of four qubits}
\address{\emph{Ivan Cheltsov}
\newline
\textnormal{University of Edinburgh,  Edinburgh, Scotland}
\newline
\textnormal{\texttt{i.cheltsov@ed.ac.uk}}}

\address{\emph{Maksym Fedorchuk}
\newline
\textnormal{Boston College, Chestnut Hill, Massachusetts}
\newline
\textnormal{\texttt{maksym.fedorchuk@bc.edu}}}

\address{\emph{Kento Fujita}
\newline
\textnormal{Osaka University, Osaka, Japan}
\newline
\textnormal{\texttt{fujita@math.sci.osaka-u.ac.jp}}}

\address{\emph{Anne-Sophie Kaloghiros}
\newline
\textnormal{Brunel University London, Uxbridge, England}
\newline
\textnormal{\texttt{anne-sophie.kaloghiros@brunel.ac.uk}}}

\begin{document}

\begin{abstract}
We find all K-polystable limits of divisors in $(\mathbb{P}^1)^4$ of degree $(1,1,1,1)$
and explicitly describe the~associated irreducible component of the~K-moduli space.
\end{abstract}

\maketitle

\section{Introduction}
\label{section:intro}

Smooth Fano 3-folds have been classified by Iskovskikh, Mori, Mukai into 105 \mbox{families}.
Fano 3-folds in each deformation family can be parametrised by an~irreducible rational variety.
For instance, the~family \textnumero~4.1 contains smooth Fano 3-folds of Picard rank $4$ and anticanonical degree $24$.
These 3-folds are smooth divisors of degree $(1,1,1,1)$ in~$(\mathbb{P}^1)^4$.
Moreover, it follows from \cite{ChterentalDjokovic,HolweckLuquePlanat,Verstraete} that
each smooth member of this deformation family can be given~in  $\mathbb{P}^1_{x_1,y_1}\times\mathbb{P}^1_{x_2,y_2}\times\mathbb{P}^1_{x_3,y_3}\times\mathbb{P}^1_{x_4,y_4}$ by the~equation
\begin{multline}
\label{equation:form-1}\tag{$\heartsuit$}
\frac{a+d}{2}\big(x_1x_2x_3x_4+y_1y_2y_3y_4\big)
+\frac{a-d}{2}\big(x_1x_2y_3y_4+y_1y_2x_3x_4\big)+\\
+\frac{b+c}{2}\big(x_1y_2x_3y_4+y_1x_2y_3x_4\big)
+\frac{b-c}{2}\big(x_1y_2y_3x_4+y_1x_2x_3y_4\big)=0
\end{multline}
for some $a,b,c,d\in\mathbb{C}$ such that $(a^2-b^2)(a^2-c^2)(a^2-d^2)(b^2-c^2)(b^2-d^2)(c^2-d^2)\ne 0$.
So, smooth Fano 3-folds in the~family \textnumero 4.1 can be parametrised by an~open subset in $\mathbb{P}^3$.
These 3-folds together with their singular limits appeared in many problems in algebraic geometry \cite{BelousovLoginov,Boozer,CheltsovDuboulozKishimoto,CheltsovMarquandTschinkelZhang,CheltsovTschinkelZhang,KuznetsovProkhorov}
and also in mathematical physics \cite{ChterentalDjokovic,DietrichGraafMarraniOriglia,GourWallach,HolweckLuquePlanat,HolweckLuqueThibon,LuqueThibon,Verstraete}.
For example, all smooth 3-folds in this family are known to be K-stable \cite{BelousovLoginov}.

On the~other hand, we know that there exists a~projective moduli space $\mathrm{M}^\mathrm{Kps}_{3,24}$
whose closed points parametrise K-polystable Fano 3-folds of (anticanonical) degree $24$ \cite{LiuXuZhuang}.
The~goal of this paper is to describe the~irreducible component of this space whose points parametrise smooth members of the~family \textnumero 4.1
together with their K-polystable limits.

\begin{maintheorem*}
Let $\mathbf{M}$ be the~irreducible component of the~moduli space $\mathrm{M}^\mathrm{Kps}_{3,24}$ whose points parametrise K-polystable limits of smooth Fano 3-folds in the~family~\textnumero~4.1.
Then
\begin{itemize}
\item $\mathbf{M}$ is the~blow-up at a~smooth point of $\mathbb{P}(1,3,4,6)$ with weights $(1,2,3)$;
\item if $X$ is a~Fano 3-fold parameterised by a~closed point in $\mathbf{M}$, then
\begin{itemize}
\item either $X$ is a~divisor in $(\mathbb{P}^1)^4$ given by \eqref{equation:form-1} such that
$$
(a:b:c:d)\not\in\big\{(1:0:0:0), (0:1:0:0), (0:0:1:0), (0:0:0:1), (1:\pm 1:\pm 1:\pm 1)\big\},
$$

\item or $X$ is a~divisor in $\mathbb{P}(1_{s_1},1_{t_1},2_{w_1})\times \mathbb{P}(1_{s_2},1_{t_2},2_{w_2})$ given by
\begin{equation}
\label{equation:form-2}\tag{$\diamondsuit$}
w_1w_2=as_1t_1s_2t_2 +\frac{b+c}{4}\big(s_1^2s_2^2+t_1^2t_2^2\big)+\frac{b-c}{4}\big(s_1^2t^2_2+t^2_1s^2_2\big)
\end{equation}
for any $(a:b:c) \in \mathbb{P}^2$.
\end{itemize}
\end{itemize}
\end{maintheorem*}

For more details about the~3-fold $\mathbf{M}$ in this theorem, see Proposition \ref{proposition:toric}.

\begin{remark}
\label{remark:Colino}
Let $X$ be one of the~K-polystable Fano 3-folds described in Main Theorem,
and let $\widetilde{X}\to X$ be its standard resolution of singularities which is given by the~ordinary blow up of the~singular locus.
Then the~following are equivalent
\begin{enumerate}
\item the~Intermediate Jacobian of the~3-fold $\widetilde{X}$ is not trivial,
\item the~Intermediate Jacobian of the~3-fold $\widetilde{X}$ is a smooth elliptic curve,
\item either $X$ is smooth or $X$ is a divisor in $\big(\mathbb{P}(1,1,2)\big)^2$ given by \eqref{equation:form-2} such that
$$
\mathrm{Sing}(X)=\big\{s_1=t_1=w_2=0\big\}\cup\big\{s_2=t_2=w_1=0\big\}.
$$
\end{enumerate}
This reminds a~similar phenomenon discovered in \cite{Collino} for degenerations of cubic 3-folds,
which is also discussed in \cite[Remark~5.6]{KollarLazaSaccaVoisin}.
\end{remark}

The structure of the~paper is the~following.
In Section~\ref{section:geometry}, we describe the~GIT moduli space of divisors of degree $(1,1,1,1)$ in $(\mathbb{P}^1)^4$,
and we describe singularities of divisors~\eqref{equation:form-1}.
In Section~\ref{section:2-2-divisors}, we study basic properties of divisors in  $\big(\mathbb{P}(1,1,2)\big)^2$ given by the~equation~\eqref{equation:form-2}.
In Section~\ref{section:K-stability}, we prove K-polystability of the~following Fano 3-folds:
\begin{itemize}
\item all irreducible divisors in $(\mathbb{P}^1)^4$ given by \eqref{equation:form-1};
\item all divisors in $\big(\mathbb{P}(1,1,2)\big)^2$ given by \eqref{equation:form-2}.
\end{itemize}
In Section~\ref{section:K-moduli}, we complete the~proof of Main Theorem.
In Section~\ref{section:stack}, we study K-semistable degenerations of smooth Fano 3-folds in the~family \textnumero 4.1.
Finally, in Appendix~\ref{section:4-lines}, we~show that singular K-polystable divisors if degree $(1,1,1,1)$ in $(\mathbb{P}^1)^4$ and their K-polystable limits
can be constructed as anticanonical models of very simple blow ups of  $\mathbb{P}^3$.

\medskip
\noindent
\textbf{Acknowledgements.}
Ivan Cheltsov has been supported by EPSRC grant  EP/V054597/1 and by Simons Collaboration grant \emph{Moduli of varieties}.
Kento Fujita has been supported by JSPS KAKENHI Grant Number 22K03269, Royal Society International Collaboration Award ICA\textbackslash 1\textbackslash 23109 and Asian Young Scientist Fellowship.
Anne-Sophie Kaloghiros has been supported by EPSRC grant EP/V056689/1.

\section{Divisors of degree $(1,1,1,1)$ in $(\mathbb{P}^1)^4$}
\label{section:geometry}

In this section, we study divisors of degree $(1,1,1,1)$ in $\mathbb{P}^1_{x_1,y_1}\times\mathbb{P}^1_{x_2,y_2}\times\mathbb{P}^1_{x_3,y_3}\times\mathbb{P}^1_{x_4,y_4}$.
Let $V=H^0\big((\mathbb{P}^1)^4, \mathcal{O}(1,1,1,1)\big)$. Set $\Gamma_0=\operatorname{SL}_2(\mathbb{C})^4$,
set
$$
\Gamma=\Gamma_0\rtimes \mathfrak{S}_4=\operatorname{SL}_2(\mathbb{C})^4\rtimes\mathfrak{S}_4,
$$
and consider the~natural actions of these two groups on the~space $V$.
In this~section, we will study the~projective GIT quotient $\mathbb{P}(V)^{\operatorname{ss}}\sslash \Gamma$.
Here, we use classical notation for the~projectivization $\mathbb{P}(V)$ --- the~variety of one-dimensional subspaces in $V$.

Note that, physicists often use the~symbol $\operatorname{SLOCC}$ in place of $\Gamma_0$,
because this reductive group is called as \emph{Stochastic Local Operations with Classical Communication}.

\subsection{Normal forms}
Forms in $V$ were studied extensively in the~context of quantum computing \cite{ChterentalDjokovic,GourWallach,HolweckLuquePlanat,HolweckLuqueThibon,LuqueThibon,Verstraete},
as they represent entanglement of systems of four~qubits.
In particular, it was shown that such forms can be put in normal forms,
which gives

\begin{corollary}[{\cite{Verstraete,ChterentalDjokovic}}]
\label{corollary:smooth}
Let $f$ be a~non-zero form in $V$ such that $\{f=0\}$ is smooth.
Then the~orbit $\Gamma_0.f$ contains the~form
\begin{multline}
\label{form:1}\tag{$G_{a,b,c,d}$}
\frac{a+d}{2}\big(x_1x_2x_3x_4+y_1y_2y_3y_4\big)
+\frac{a-d}{2}\big(x_1x_2y_3y_4+y_1y_2x_3x_4\big)+\\
+\frac{b+c}{2}\big(x_1y_2x_3y_4+y_1x_2y_3x_4\big)
+\frac{b-c}{2}\big(x_1y_2y_3x_4+y_1x_2x_3y_4\big)
\end{multline}
for some $(a,b,c,d)\in\mathbb{C}^4$.
\end{corollary}

Now, we consider the~following vector subspace:
$$
\mathcal{G}=\big\{ f\in V\ \vert\ f\ \text{is the~form \eqref{form:1}} \text{ for some } (a,b,c,d) \in \mathbb{C}^4\big\}\subset V.
$$
Let $G$ be the~subgroup in $\operatorname{Aut}\left((\mathbb{P}^1)^4\right)$ generated by the~following involutions:
\begin{align*}
\tau_1\colon \big((x_1:y_1),(x_2:y_2),(x_3:y_3),(x_4:y_4)\big)&\mapsto\big((x_2:y_2),(x_1:y_1),(x_4:y_4),(x_3:y_3)\big),\\
\tau_2\colon \big((x_1:y_1),(x_2:y_2),(x_3:y_3),(x_4:y_4)\big)&\mapsto\big((x_4:y_4),(x_3:y_3),(x_2:y_2),(x_1:y_1)\big),\\
\tau_3\colon \big((x_1:y_1),(x_2:y_2),(x_3:y_3),(x_4:y_4)\big)&\mapsto\big((y_1:x_1),(y_2:x_2),(y_3:x_3),(y_4:x_4)\big),\\
\tau_3\colon \big((x_1:y_1),(x_2:y_2),(x_3:y_3),(x_4:y_4)\big)&\mapsto\big(((x_1:-y_1),(x_2:-y_2),(x_3:-y_3),(x_4:-y_4)\big).
\end{align*}
Then $G\simeq(\mathbb{Z}/2\mathbb{Z})^4$, every form in $\mathcal{G}$ is $G$-invariant, $\mathcal{G}$ contains all $G$-invariant forms in $V$,
and $\mathcal{G}$ contains defining polynomials of all smooth divisors of degree $(1,1,1,1)$ in $(\mathbb{P}^1)^4$.

For sufficiently general form $f\in\mathcal{G}$, the~$G$-invariant divisor $\{f=0\}\subset (\mathbb{P}^1)^4$~is~smooth.
The following lemma describes reducible divisors $\{f=0\}$ for $f\in\mathcal{G}$.

\begin{lemma}
\label{lemma:reducible}
Let $f$ be a non-zero form in $\mathcal{G}$. Then
$$
f \mbox{ is reducible } \iff
\begin{cases}
a^2=b^2=c^2=d^2 \mbox{ or } \\
a=b=c=0 \mbox{ or } \\
a=b=d=0 \mbox{ or } \\
a=c=d=0 \mbox{ or } \\
b=c=d=0
\end{cases} \iff f\in \Gamma.(x_1x_2-y_1y_2)(x_3x_4-y_3y_4),
$$
where $(x_1x_2-y_1y_2)(x_3x_4-y_3y_4)$ is the~form \eqref{form:1} with $(a,b,c,d)=(0,0,0,1)$.
\end{lemma}

\begin{proof}
Direct computation.
\end{proof}

To study the~GIT quotient $\mathbb{P}(V)^{\operatorname{ss}}\sslash \Gamma$, we need the~following lemma.

\begin{lemma}
\label{lemma:aut}
Let $f$ and $f'$ be non-zero forms in $\mathcal{G}$, let $X=\{ f=0\}$ and $X'=\{f'=0\}$.
Then $X$ and $X'$ are isomorphic if and only if $f$ and $f'$ lie in one $\Gamma$-orbit.
\end{lemma}

\begin{proof}
By Lemma~\ref{lemma:reducible}, we may assume that both $X$ and $X'$ are irreducible.
In fact, they are normal by Theorem~\ref{theorem:sing}, but we do not need this.
Then the~restriction homomorphism
$$
\operatorname{Pic}\left((\mathbb{P}^1)^4\right)\to\operatorname{Pic}(X)
$$
is an~isomorphism by the~Lefschetz hyperplane section theorem.

For $i\in\{1,2,3,4\}$, let $\operatorname{pr}^i \colon X \to (\mathbb{P}^1)^3$ be the~projection to all but the~$i^{\mathrm{th}}$ factor~of~$(\mathbb{P}^1)^4$.
Then $\operatorname{pr}^i$ is birational, so the~Mori cone $\overline{\operatorname{NE}(X)}$ is spanned by all $\operatorname{pr}^i$-exceptional curves.
Since $\operatorname{Nef}(X)$ is the~dual cone of $\overline{\operatorname{NE}(X)}$, it is spanned by the~pullbacks
$$
\operatorname{pr}_1^*(\mathcal{O}_{\mathbb{P}^1}(1)),\quad \operatorname{pr}_2^*(\mathcal{O}_{\mathbb{P}^1}(1)), \quad \operatorname{pr}_3^*(\mathcal{O}_{\mathbb{P}^1}(1)),\quad  \operatorname{pr}_4^*(\mathcal{O}_{\mathbb{P}^1}(1)),
$$
where $\operatorname{pr}_i\colon X\to \mathbb{P}^1$ is the~projection to the~$i^{\mathrm{th}}$ factor of $(\mathbb{P}^1)^4$.
Using this, we see that any isomorphism $X\simeq X'$ is induced by an~automorphism of $(\mathbb{P}^1)^4$.
\end{proof}

We now consider the~action of $\Gamma_0$ and $\Gamma$ on forms in $\mathcal{G}$. For convenience, we set
$$
\begin{cases}
u_1= x_1x_2x_3x_4+ x_1x_2y_3y_4+y_1y_2x_3x_4+ y_1y_2y_3y_4,\\
u_2= x_1x_2x_3x_4- x_1x_2y_3y_4-y_1y_2x_3x_4+ y_1y_2y_3y_4,\\
u_3=x_1y_2x_3y_4+ x_1y_2y_3x_4+y_1x_2x_3y_4+ y_1x_2y_3x_4,\\
u_4= x_1y_2x_3y_4- x_1y_2y_3x_4-y_1x_2x_3y_4+ y_1x_2y_3x_4.
\end{cases}
$$
Then, we can rewrite the~form \eqref{form:1} in $\mathcal{G}$ as
$$
\frac{a}{2}u_1+\frac{d}{2}u_2+\frac{b}{2}u_3+\frac{c}{2}u_4.
$$
Let $\Gamma_0^\mathcal{G}$ and $\Gamma^\mathcal{G}$ be the~images in $\operatorname{GL}(\mathcal{G})$ of the~subgroups of the~groups $\Gamma_0$
and $\Gamma$ that leave the~subspace $\mathcal{G}$ invariant, respectively.
By \cite[Section~II]{GourWallach}, the~group $\Gamma_0^\mathcal{G}$ is generated~by
\begin{align*}
(u_1,u_2,u_3,u_4)&\mapsto (-u_1,-u_2,u_3,u_4),\\
(u_1,u_2,u_3,u_4)&\mapsto (u_{2},u_{1},u_3,u_4),\\
(u_1,u_2,u_3,u_4)&\mapsto (u_{2},u_{3},u_4,u_1),
\end{align*}
which implies that $\Gamma_0^\mathcal{G}\simeq\big(\mathbb{Z}/2\mathbb{Z}\big)^3\rtimes\mathfrak{S}_4$, and $\Gamma_0^\mathcal{G}$ is isomorphic to the~Weyl group $W(\mathrm{D}_4)$.
Similarly, the~group $\Gamma^\mathcal{G}$ is generated by $\Gamma_0^\mathcal{G}$, the~transformation
$$
(u_1,u_2,u_3,u_4)\mapsto (u_1,u_2,u_3,-u_4)
$$
and the~transformation
$$
(u_1,u_2,u_3,u_4)\mapsto \Big(\frac{u_1+u_2+u_3+u_4}{2},\frac{u_1+u_2-u_3-u_4}{2},\frac{u_1-u_2+u_3-u_4}{2},\frac{u_1-u_2-u_3+u_4}{2}\Big).
$$
The group $\Gamma^\mathcal{G}$ is isomorphic to the~Weyl group $W(\mathrm{F}_4)$, and its GAP ID is [1152,157478].
Since $-1\in \Gamma^\mathcal{G}$ acts trivially~on~$\mathbb{P}(\mathcal{G})$,
we obtain an~effective action of the~group
$$
\Gamma^\mathcal{G}/\langle -1\rangle\simeq W(\mathrm{F}_4)/\langle -1\rangle\simeq\big(\mathfrak{A}_4\times\mathfrak{A}_4\big)\rtimes\big(\mathbb{Z}/2\mathbb{Z}\big)^2
$$
on $\mathbb{P}(\mathcal{G})\simeq\mathbb{P}^3_{a,b,c,d}$. The GAP ID of the~group $W(\mathrm{F}_4)/\langle -1\rangle$ is [576,8654].

In the~following, we identify $\Gamma^\mathcal{G}=W(\mathrm{F}_4)$, $\Gamma^\mathcal{G}/\langle -1\rangle=W(\mathrm{F}_4)/\langle -1\rangle$ and $\mathbb{P}(\mathcal{G})=\mathbb{P}^3_{a,b,c,d}$.

\subsection{GIT moduli space}
\label{GIT}
Now, we recall some results on the~GIT quotient $\mathbb{P}(V)^{\operatorname{ss}}\sslash \Gamma$.
Recall~that $V=H^0\big((\mathbb{P}^1)^4, \mathcal{O}(1,1,1,1)\big)$,
and fix its basis: for $i,j,k,l\in\{0,1\}$, we let
$$
a_m=x_1^iy_1^{1-i}x_2^jy_2^{1-j}x_3^ky_3^{1-k}x_4^ly_4^{1-l}
$$
for $m=8i+4j+2k+l$. Therefore, for example, we have $a_0= x_1x_2x_3x_4$ and $a_3= x_1x_2y_3y_4$.
Let $W=V^\vee$, and let $\mathcal{S}$ be the~symmetric algebra of $W$.
Then it follows from \cite{LuqueThibon}~that the~$\Gamma_0$-invariant subring of $\mathcal S$ is $\mathcal{S}^{\Gamma_0}=\mathbb{C}[H,L,M,D]$,
where
$$
H=a_0a_{15}-a_1a_{14}-a_2a_{13}+a_3a_{12}-a_4a_{11}+a_5a_{10}+a_6a_9-a_7a_8,
$$
$$
L=\begin{vmatrix}
a_0 & a_4 & a_8 & a_{12}\\
a_1 & a_5 & a_9 & a_{13}\\
a_2 & a_6 & a_{10} & a_{14}\\
a_3 & a_7 & a_{11} & a_{15}
\end{vmatrix}, \quad
M=\begin{vmatrix}
a_0 & a_8 & a_2 & a_{10}\\
a_1 & a_9 & a_3 & a_{11}\\
a_4 & a_{12} & a_6 & a_{14}\\
a_5 & a_{13} & a_7 & a_{15}
\end{vmatrix}
$$
and
\begin{multline*}
D=\small{\begin{vmatrix}
a_0a_9-a_1a_8 & a_0a_{11}+a_2a_9-a_1a_{10}-a_3a_8 & a_2a_{11}-a_3a_{10}\\
a_0a_{13}+a_4a_9-a_1a_{12}-a_5a_8 &
\begin{matrix}(a_0a_{15}+a_2a_{13}+a_4a_{11}+a_6a_9\\
-a_1a_{14}-a_3a_{12}-a_5a_{10}-a_7a_8)\end{matrix} &
a_2a_{15}+a_6a_{11}-a_3a_{14}-a_7a_{10}\\
a_4a_{13}-a_5a_{12} & a_4a_{15}+a_6a_{13}-a_5a_{14}-a_7a_{12} &
a_6a_{15}-a_7a_{14}
\end{vmatrix}.}
\end{multline*}
To describe $\mathcal{S}^\Gamma=\left(\mathcal{S}^{\Gamma_0}\right)^{\mathfrak{S}_4}$, let $R, S$ and $T$ be the~forms of degrees $6$, $8$ and $12$ given by
\begin{equation}
\label{RST}
\begin{split}
R&=H(L-M)+3D, \\
S&=\frac{1}{12}H^4-\frac{2}{3}H^2L+\frac{2}{3}H^2M-2HD+\frac{4}{3}(L^2+LM+M^2), \\
T&=\frac{1}{216}H^6-\frac{1}{18}H^4(L-M)-\frac{1}{6}H^3D+\frac{1}{9}H^2(2L^2-LM+2M^2)\\
&\quad\quad+\frac{2}{3}H(L-M)D-\frac{8}{27}(L^3-M^3)-\frac{4}{9}LM(L-M)+D^2.
\end{split}
\end{equation}
Observe that $H$, $R$, $S$, $T$ are algebraically independent.

\begin{proposition}[\cite{Schlafli,ChterentalDjokovic,GourWallach}]\label{proposition:HRST}
One has $\mathcal{S}^\Gamma=\mathbb{C}[H,R,S,T]$,
so that
$$
\mathbb{P}(V)^{\operatorname{ss}}\sslash \Gamma\simeq\mathbb{P}(1,3,4,6).
$$
\end{proposition}

\begin{proof}
Let us give a~proof for readers' convenience.
Under the~natural embedding
$$
\mathfrak{S}_3=\{\operatorname{id}, (1,2,3), (1,3,2), (1,2), (1,3), (2,3)\}\subset\mathfrak{S}_4,
$$
we have $\mathfrak{S}_4=(\mathbb{Z}/2\mathbb{Z})^2\rtimes\mathfrak{S}_3$, where $(\mathbb{Z}/2\mathbb{Z})^2\triangleleft \mathfrak{S}_4$
is the~Klein four-group
$$
\{\operatorname{id}, (1,2)(3,4), (1,3)(2,4), (1,4)(2,3)\}.
$$
We can directly check the~following:
\begin{itemize}
\item
We have $H^\sigma=H$ and $R^\sigma=R$ for every $\sigma\in\mathfrak{S}_4$.
\item
We have $L^\sigma=L$, $M^\sigma=M$ and $D^\sigma=D$ for every $\sigma\in(\mathbb{Z}/2\mathbb{Z})^2$.
\end{itemize}
Thus, we have $\left(\mathcal{S}^{\Gamma_0}\right)^{(\mathbb{Z}/2\mathbb{Z})^2}=\mathbb{C}[H,L,M,D]$ and $\mathcal{S}^\Gamma=\left(\mathcal{S}^{\Gamma_0}\right)^{\mathfrak{S}_3}$.

Observe that the~subring $\mathbb{C}[L,M]\subset\mathbb{C}[H,L,M,D]$ is $\mathfrak{S}_3$-invariant, since
$$
\begin{matrix}
L^{(1,2)}=-L, & L^{(1,3)}=M-L, & L^{(2,3)}=-M, & L^{(1,2,3)}=L-M, & L^{(1,3,2)}=M,\\
M^{(1,2)}=M-L, & M^{(1,3)}=-M, & M^{(2,3)}=-L, & M^{(1,2,3)}=L, & M^{(1,3,2)}=L-M.
\end{matrix}
$$
Note that the~polynomials $L^2+LM+M^2$ and $2L^3+3L^2M-3LM^2-2M^3$ are $\mathfrak{S}_3$-invariant.
Moreover, the~Hilbert polynomial of the~invariant ring $\mathbb{C}[L,M]^{\mathfrak{S}_3}$ is equal to
$$
\frac{1}{6}\left(\frac{1}{(1-t)^2}+\frac{3}{1-t^2}+\frac{2}{1+t+t^2}\right)=\frac{1}{(1-t^2)(1-t^3)}
$$
by Molien's formula, where we take the~degrees of $L$ and $M$ to be $1$. This gives
$$
\mathbb{C}[L,M]^{\mathfrak{S}_3}=\mathbb{C}\left[L^2+LM+M^2, 2L^3+3L^2M-3LM^2-2M^3\right],
$$
because $L^2+LM+M^2$ and $2L^3+3L^2M-3LM^2-2M^3$ are algebraically independent. This shows that
$$
\mathbb{C}[H,L,M,D]^{\mathfrak{S}_3}=\mathbb{C}\left[H,R,L^2+LM+M^2,2L^3+3L^2M-3LM^2-2M^3\right].
$$
On the~other hand, the~equality
$$
\mathbb{C}\left[H,R,L^2+LM+M^2,2L^3+3L^2M-3LM^2-2M^3\right]=\mathbb{C}[H,R,S,T]
$$
is trivial from the~definitions of $S$ and $T$.
\end{proof}

Let $\phi\colon\mathbb{P}(V)^{\operatorname{ss}}\to\mathbb{P}(V)^{\operatorname{ss}}\sslash\Gamma$ be the~GIT quotient morphism.
Then $\mathbb{P}(\mathcal{G})\subset \mathbb{P}(V)^{\operatorname{ss}}$ by

\begin{proposition}[{\cite{ChterentalDjokovic}}]
\label{proposition:ChterentalDjokovic}
The subset $\Gamma.\mathbb{P}(\mathcal{G})\subset\mathbb{P}(V)$ consists of all GIT polystable points with respects to the~action $\Gamma\curvearrowright V$.
\end{proposition}

Thus, we can consider the~restriction of $\phi$ to $\mathbb{P}(\mathcal{G})$:
$$
\phi\big|_{\mathbb{P}(\mathcal{G})}\colon\mathbb{P}(\mathcal{G})=\mathbb{P}^3_{a,b,c,d}\to\mathbb{P}(V)^{\operatorname{ss}}\sslash\Gamma\simeq\mathbb{P}(1,3,4,6).
$$
In the~following, we identify $\mathbb{P}(V)^{\operatorname{ss}}\sslash\Gamma=\mathbb{P}(1,3,4,6)$.

\begin{proposition}
\label{proposition:explicit}
The restriction $\phi|_{\mathbb{P}(\mathcal{G})}$ is the~quotient of $\mathbb{P}(\mathcal{G})$ by the~group $W(\mathrm{F}_4)/\langle -1\rangle$,
which is the~composition $\phi_3\circ\phi_2\circ\phi_1$, where
$\phi_1\colon{\mathbb{P}}^3\to\mathbb{P}^3$ is a~degree $8$ morphism given by
$$
(a:b:c:d)\mapsto\big(a^2:b^2:c^2:d^2\big),
$$
the morphism $\phi_2\colon\mathbb{P}^3\to\mathbb{P}(1,2,3,4)$ is a~morphism of degree $24$ given by
\begin{multline*}
(p_1:p_2:p_3:p_4)\mapsto\bigl(p_1+p_2+p_3+p_4:p_1p_2+p_1p_3+p_1p_4+p_2p_3+p_2p_4+p_3p_4:\\
:p_2p_3p_4+p_1p_3p_4+p_1p_2p_4+p_1p_2p_3:p_1p_2p_3p_4\bigr),
\end{multline*}
and $\phi_3\colon\mathbb{P}(1,2,3,4)\to\mathbb{P}(1,3,4,6)$ is a~morphism of degree $3$ given by
\begin{multline*}
(s_1:s_2:s_3:s_4)\mapsto\bigg(\frac{1}{2}s_1:\frac{1}{32}\left(s_1^3-4s_1s_2+24s_3\right):\frac{1}{12}\left(12s_4+s_2^2-3s_1s_3\right):\\
:\frac{1}{432}\left(27s_1^2s_4-72s_2s_4+2s_2^3-9s_1s_2s_3+27s_3^2\right)\bigg)
\end{multline*}
where $s_1,s_2,s_3,s_4$ are coordinates on $\mathbb{P}(1,2,3,4)$ of weights $1$, $2$, $3$, $4$, respectively.
\end{proposition}

\begin{proof}
Direct computations.
\end{proof}

For any non-zero $f\in \mathcal{G}$ that is the~form \eqref{form:1} for $(a,b,c,d)\in\mathbb{C}^4$, set
$$
X_{(a:b:c:d)}=\big\{f=0\big\}\subset (\mathbb{P}^1)^4.
$$
For convenience, we set $\mathbb{P}=\mathbb{P}^3_{a,b,c,d}$ and
\begin{align*}
\mathbb{P}^{\operatorname{Sing}}&=\left\{(a,b,c,d)\in\mathbb{P}\ |\ X_{(a:b:c:d)}\text{ is singular}\right\},\\
\mathbb{P}^{\operatorname{Red}}&=\left\{(a,b,c,d)\in\mathbb{P}\ |\ X_{(a:b:c:d)}\text{ is reducible}\right\},\\
\mathbb{P}^{\operatorname{Curv}}&=\left\{(a,b,c,d)\in\mathbb{P}\ |\ \dim(\operatorname{Sing}(X_{(a:b:c:d)}))=1\right\},\\
\mathbb{P}^{6\mathbb{A}_1}&=\left\{(a,b,c,d)\in\mathbb{P}\ |\ X_{(a:b:c:d)}\text{ has exactly $6$ isolated $\mathbb{A}_1$ singularities}\right\},\\
\mathbb{P}^{4\mathbb{A}_1}&=\left\{(a,b,c,d)\in\mathbb{P}\ |\ X_{(a:b:c:d)}\text{ has exactly $4$ isolated $\mathbb{A}_1$ singularities}\right\},\\
\mathbb{P}^{2\mathbb{A}_1}&=\left\{(a,b,c,d)\in\mathbb{P}\ |\ X_{(a:b:c:d)}\text{ has exactly $2$ isolated $\mathbb{A}_1$ singularities}\right\}.
\end{align*}

\begin{theorem}
\label{theorem:sing}
The boundary of the~moduli space $\mathbb{P}^{\operatorname{Sing}}$ admits a~stratification
\begin{equation*}
     \mathbb{P}^{\operatorname{Sing}}
=\mathbb{P}^{\operatorname{Red}}
\sqcup
\mathbb{P}^{\operatorname{Curv}}
\sqcup
\mathbb{P}^{6\mathbb{A}_1}
\sqcup
\mathbb{P}^{4\mathbb{A}_1}
\sqcup
\mathbb{P}^{2\mathbb{A}_1},
\end{equation*}
and the~closure of the~strata satisfy:
$$
\overline{\mathbb{P}^{2\mathbb{A}_1}} = \mathbb{P}^{\operatorname{Sing}}, \quad \overline{\mathbb{P}^{\operatorname{Curv}}}\setminus \mathbb{P}^{\operatorname{Curv}}= \mathbb{P}^{\operatorname{Red}}, \quad \overline{\mathbb{P}^{\operatorname{4\mathbb{A}_1}}}\setminus \mathbb{P}^{\operatorname{4\mathbb{A}_1}}= \mathbb{P}^{\operatorname{Red}} \sqcup \mathbb{P}^{\operatorname{6\mathbb{A}_1}}.
$$
Moreover, in the~notation of \eqref{RST}, the~following assertions hold:
\begin{enumerate}
\item[$(\mathrm{1})$] $\phi\big(\mathbb{P}^{\operatorname{Sing}}\big)$ is the~hypersurface
$$
\{S^3-27T^2=0\} \subset \mathbb{P}(1,3,4,6),
$$
and the~divisor $X_{(a:b:c:d)}$ is singular if and only if
$$
(a^2-b^2)(a^2-c^2)(a^2-d^2)(b^2-c^2)(b^2-d^2)(c^2-d^2)=0,
$$

\item[$(\mathrm{2})$] $\phi\big(\mathbb{P}^{\operatorname{Red}}\big)=(2:2:0:0)\in\mathbb{P}(1,3,4,6)$,
and $X_{(a:b:c:d)}$ is reducible if and only if
$$
a^2=b^2=c^2=d^2 \mbox{ or } \sigma(b)=\sigma(c)=\sigma(d)=0 \mbox{ for some } \sigma \in \mathfrak{S}_4,
$$

\item[$(\mathrm{3})$] $\phi\big(\mathbb{P}^{6\mathbb{A}_1}\big)=\Big(2:0:\frac{4}{3}:\frac{8}{27}\Big)$, and $X_{(a:b:c:d)}$ has $6$ ordinary double points if and only if
$$
\sigma(a)=\sigma(b)=0 \mbox{ and } \sigma(c)^2=\sigma(d)^2 \mbox{ for some } \sigma \in \mathfrak{S}_4.
$$

\item[$(\mathrm{4})$] $\phi\big(\overline{\mathbb{P}^{\operatorname{Curv}}}\big)=\{S=T=0\}$,
and $X_{(a:b:c:d)}$ is singular along a~curve if and only~if
$$
\sigma(b)^2=\sigma(c)^2=\sigma(d)^2\ne 0  \mbox{ for some } \sigma \in \mathfrak{S}_4,
$$

\item[$(\mathrm{5})$] $\phi\big(\mathbb{P}^{4\mathbb{A}_1}\big)$ is the~set of points
$$
\left\{\left(\frac{1}{2}(t+1):\frac{1}{32}(t+1)(t-1)^2:\frac{1}{12}t^2:\frac{1}{216}t^3\right)\ \big|\ t\in\mathbb{C}\setminus\{0,1\}\right\}\subset \mathbb{P}(1,3,4,6),
$$
and  $(a:b:c:d)\in\overline{\mathbb{P}^{4\mathbb{A}_1}}$ if and only if there is $\sigma\in\mathfrak{S}_4$ such that
either $\sigma(a)^2=\sigma(b)^2$ and $\sigma(c)^2=\sigma(d)^2$, or $\sigma(a)=\sigma(b)=0$.
\end{enumerate}
\end{theorem}

\begin{proof}
Follows from direct computation, see \cite{HolweckLuquePlanat}.
\end{proof}

Let us translate Theorem~\ref{theorem:sing} into a~simpler form. To do this, consider the~plane
$$
\{a+b=0\}\subset\mathbb{P}(\mathcal{G})=\mathbb{P}^3_{a,b,c,d}.
$$
If $(a:b:c:d)\in\{a+b=0\}$, then $X_{(a:b:c:d)}$ is singular.
Vise versa, if $X_{(a:b:c:d})$ is singular, then the~plane $\{a+b=0\}$ contains a~point in the~$W(\mathrm{F}_4)/\langle -1\rangle$-orbit of $(a:b:c:d)$.
Therefore, in Theorem~\ref{theorem:sing}, we can assume that $a+b=0$ up to the~action of $W(\mathrm{F}_4)/\langle -1\rangle$,
and we can restate the~result in terms of points in the~plane $\{a+b=0\}$. Namely, set
\begin{align*}
\mathbb{P}^{\operatorname{Red}}_{a+b=0}=\mathbb{P}^{\operatorname{Red}}\cap\{a+b=0\},\\
\mathbb{P}^{\operatorname{Curv}}_{a+b=0}=\mathbb{P}^{\operatorname{Curv}}\cap\{a+b=0\},\\
\mathbb{P}^{6\mathbb{A}_1}_{a+b=0}=\mathbb{P}^{6\mathbb{A}_1}\cap\{a+b=0\},\\
\mathbb{P}^{4\mathbb{A}_1}_{a+b=0}=\mathbb{P}^{4\mathbb{A}_1}\cap\{a+b=0\},\\
\mathbb{P}^{2\mathbb{A}_1}_{a+b=0}=\mathbb{P}^{4\mathbb{A}_1}\cap\{a+b=0\}.
\end{align*}

\begin{corollary}
\label{corollary:singular}
The plane $\{a+b=0\}\subset \mathbb{P}(\mathcal{G})=\mathbb{P}^3_{a,b,c,d}$ admits the~stratification
$$
\{a+b=0\}=\mathbb{P}^{\operatorname{Red}}_{a+b=0}\sqcup\mathbb{P}^{\operatorname{Curv}}_{a+b=0}\sqcup\mathbb{P}^{6\mathbb{A}_1}_{a+b=0}\sqcup\mathbb{P}^{4\mathbb{A}_1}_{a+b=0}\sqcup\mathbb{P}^{2\mathbb{A}_1}_{a+b=0},
$$
which satisfies
$$
\overline{\mathbb{P}^{2\mathbb{A}_1}_{a+b=0}}=\{a+b=0\}, \quad
\overline{\mathbb{P}^{\operatorname{Curv}}_{a+b=0}}\setminus
\mathbb{P}^{\operatorname{Curv}}_{a+b=0}=\mathbb{P}^{\operatorname{Red}}_{a+b=0}, \quad
\overline{\mathbb{P}^{4\mathbb{A}_1}_{a+b=0}}\setminus
\mathbb{P}^{4\mathbb{A}_1}_{a+b=0}=\mathbb{P}^{\operatorname{Red}}_{a+b=0}\sqcup\mathbb{P}^{6\mathbb{A}_1}_{a+b=0}.
$$
Moreover, under the~canonical isomorphism $\{a+b=0\}\simeq\mathbb{P}^2_{a,c,d}$,
the following holds:
\begin{align*}
\overline{\mathbb{P}^{\operatorname{Curv}}_{a+b=0}}&=\big\{(a+c)(a-c)(a+d)(a-d)=0\big\}, \\
\overline{\mathbb{P}^{4\mathbb{A}_1}_{a+b=0}}&=\big\{a(c+d)(c-d)=0\big\}, \\
\mathbb{P}^{\operatorname{Red}}_{a+b=0}&=\big\{ (1:1:1), (-1:1:1), (1:-1:1), (1:1:-1), (0:1:0), (0:0:1) \big\}, \\
\mathbb{P}^{6\mathbb{A}_1}_{a+b=0}&=\big\{ (1:0:0), (0:1:1), (0:1:-1)\big\}.
\end{align*}
\end{corollary}

Now, up to the~action of the~group $W(\mathrm{F}_4)/\langle -1\rangle$ on $\mathbb{P}(\mathcal{G})=\mathbb{P}^3_{a,b,c,d}$,
we can explicitly describe all singular divisors $X_{(a:b:c:d)}\subset(\mathbb{P}^1)^4$ as follows:
\begin{itemize}
\item the~divisor $X_{(1:-1:1:1)}$ is reducible and is given by
$$
(x_1x_4-y_1y_4)(x_2x_3-y_2y_3)=0,
$$
so $X_{(1:-1:1:1)}$ is a~union of two divisors of degree $(1,0,0,1)$ and $(0,1,1,0)$;

\item the~divisor $X_{(0:0:1:1)}$ has $6$ isolated $\mathbb{A}_1$ singularities at
\begin{align*}
\big((1:-1),(1:-1),(1:1),(1:1)\big), &\quad \big((1:1),(1:1),(1:1),(1:1)\big), \\
\big((1:1),(1:-1),(1:-1),(1:1)\big), &\quad \big((1:-1),(1:-1),(1:1),(1:-1)\big), \\
\big((1:-1),(1:-1),(1:-1),(1:-1)\big), &\quad \big((1:1),(1:1),(1:-1),(1:-1)\big),
\end{align*}
and $X_{(0:0:1:1)}$ is toric --- the~toric K-polystable Fano 3-fold with reflexible ID 1530
in the~Graded Ring Data Base of toric canonical Fano 3-folds \cite{Kasprzyk};

\item the~divisor $X_{(1:-1:1:d)}$ with $d\neq \pm 1$ is singular along the~degree $(1,1,1,1)$ curve
$$
\left\{\big((u:v),(v:u),(u:v),(v:u)\ \big|\ (u:v)\in\mathbb{P}^1\right\}\subset\big(\mathbb{P}^1\big)^4,
$$
and $X_{(1:-1:1:d)}$ has $\mathbb{A}_1$ singularity along this curve;

\item  the~divisor $X_{(0:0:1:d)}$ with $d\neq \pm 1$ and $d\neq 0$ has $4$ isolated $\mathbb{A}_1$ singularities
\begin{align*}
\big((1:-1),(1:-1),(1:1),(1:1)\big), &\quad \big((1:1),(1:1),(1:1),(1:1)\big), \\
\big((1:-1),(1:-1),(1:-1),(1:-1)\big), &\quad \big((1:1),(1:1),(1:-1),(1:-1)\big),
\end{align*}
\item general divisor $X_{(a:-a:c:d)}$ has $2$ isolated $\mathbb{A}_1$ singularities at
\begin{align*}
\big((1:-1),(1:-1),(1:1),(1:1)\big), &\quad \big((1:1),(1:1),(1:1),(1:1)\big).
\end{align*}
\end{itemize}

\section{Divisors of degree $(2,2)$ in $\mathbb{P}(1,1,2)\times\mathbb{P}(1,1,2)$}
\label{section:2-2-divisors}

In Section~\ref{section:geometry}, we described all GIT polystable degenerations of smooth Fano~3-folds in the~deformation family \textnumero 4.1.
We will prove later in Section~\ref{section:K-stability}, that all these degenerations are K-polystable provided that they are irreducible.
However, one point of the~GIT moduli space corresponds to the~orbit of the~divisor
$$
\big\{(x_1x_2-y_1y_2)(x_3x_4-y_3y_4)=0\big\}\subset\mathbb{P}^1_{x_1,y_1}\times\mathbb{P}^1_{x_2,y_2}\times\mathbb{P}^1_{x_3,y_3}\times\mathbb{P}^1_{x_4,y_4},
$$
which is a~union of divisors of degree $(1,1,0,0)$ and $(0,0,1,1)$.
It is not K-polystable, and we have to resolve this issue. This can be done as follows.

In the~notations of Section~\ref{section:geometry}, the~GIT moduli space is the~global quotient
$$
\mathbb{P}^3_{a,b,c,d}\slash \big(W(\mathrm{F}_4)\big/\langle -1\rangle\big),
$$
and the~reducible divisor corresponds to the~orbit of the~point $(0:0:0:1)\in\mathbb{P}^3_{a,b,c,d}$.
When we approach the~point $(0:0:0:1)$ along a~line in $\mathbb{P}^3_{a,b,c,d}$, we can reparametrise the~corresponding one-dimensional
subfamily of divisors $X_{(a:b:c:d)}\subset(\mathbb{P}^1)^4$ and find their unique K-polystable limit that replaces the~reducible one.
To find such reparametrisations, we consider the~embedding $X_{(a:b:c:d)}\hookrightarrow\mathbb{P}^3\times\mathbb{P}^3$ via the~map
\begin{multline*}
\quad\quad\quad\quad\big((x_1:y_1),(x_2:y_2),(x_3:y_3),(x_4:y_4)\big)\mapsto\\
\mapsto\big((x_1x_2-y_1y_2:2x_1y_2:2x_2y_1:x_1x_2+y_1y_2),(x_3x_4-y_3y_4:2x_3y_4:2x_4y_3:x_3x_4+y_3y_4)\big).
\end{multline*}
Then the~image of $X_{(a:b:c:d)}\subset\mathbb{P}^3_{u_1,u_2,u_3,u_4}\times\mathbb{P}^3_{v_1,v_2,v_3,v_4}$ is given by
$$
\left\{\aligned
&u_1^2+u_2u_3-u_4^2=0, \\
&v_1^2+v_2v_3-v_4^2=0, \\
&du_1v_1-\frac{b}{4}(u_2+u_3)(v_2+v_3)-\frac{c}{4}(u_2-u_3)(v_2-v_3)-au_4v_4=0.
\endaligned
\right.
$$
Now, if we approach the~point $(0:0:0:1)\in\mathbb{P}^3_{a,b,c,d}$ along the~line
$$
L_{(a:b:c)}=\big\{(sa:sb:sc:t)\ |\ (s:t)\in\mathbb{P}^1\big\}\subset\mathbb{P}^3_{a,b,c,d},
$$
we obtain the~one parameter family
$$
\left\{\aligned
&u_1^2+u_2u_3-u_4^2=0, \\
&v_1^2+v_2v_3-v_4^2=0, \\
&u_1v_1-s\Bigg(\frac{b}{4}(u_2+u_3)(v_2+v_3)-\frac{c}{4}(u_2-u_3)(v_2-v_3)-au_4v_4\Bigg)=0,
\endaligned
\right.
$$
which degenerates to a~reducible 3-fold when $s$ approaches $0$.
We can reparametrise it by changing the~parameter $s\mapsto s^2$,
and then changing coordinates as $u_1\mapsto su_1$, $v_1\mapsto sv_1$,
which gives the~one-parameter family
$$
\left\{\aligned
&s^2u_1^2+u_2u_3-u_4^2=0, \\
&s^2v_1^2+v_2v_3-v_4^2=0, \\
&u_1v_1-\frac{b}{4}(u_2+u_3)(v_2+v_3)-\frac{c}{4}(u_2-u_3)(v_2-v_3)-au_4v_4=0.
\endaligned
\right.
$$
When $s$ approaches $0$, the~3-folds in this family degenerate to the~singular 3-fold
\begin{equation}
\label{equation:limit}
\left\{\aligned
&u_2u_3-u_4^2=0, \\
&v_2v_3-v_4^2=0, \\
&u_1v_1-\frac{b}{4}(u_2+u_3)(v_2+v_3)-\frac{c}{4}(u_2-u_3)(v_2-v_3)-au_4v_4=0.
\endaligned
\right.
\end{equation}
Note that $\mathbb{P}(1_{s_1},1_{t_1},2_{w_1})\times\mathbb{P}(1_{s_2},1_{t_2},2_{w_2})\simeq\{u_2u_3-u_4^2=0,v_2v_3-v_4^2=0\}$ via the~map
$$
\big((s_1,t_1,w_1),(s_2,t_2,w_1)\big)\mapsto\big((w_1,s_1^2,t_1^2,s_1t_1),(w_2,s_2^2,t_2^2,s_2t_2)\big),
$$
so \eqref{equation:limit} is isomorphic to the~3-fold in $\mathbb{P}(1_{s_1},1_{t_1},2_{w_1})\times\mathbb{P}(1_{s_2},1_{t_2},2_{w_2})$ given by~\eqref{equation:form-2}.

For every $(a:b:c)\in\mathbb{P}^2$, the~Fano 3-fold \eqref{equation:form-2} has canonical Gorenstein singularities.
Moreover, we will see in Section~\ref{section:K-moduli} that this 3-fold is K-polystable for every $(a:b:c)\in\mathbb{P}^2$.

\begin{remark}
A more conceptual and fancy explanation of what we just did if the~following. Consider the~family of complete intersections in $\mathbb{P}^3_{u_1,u_2,u_3,u_4}\times\mathbb{P}^3_{v_1,v_2,v_3,v_4}$ given by
$$
\left\{\aligned
&d^2u_1^2+u_2u_3-u_4^2=0, \\
&d^2v_1^2+v_2v_3-v_4^2=0, \\
&d^2u_1v_1-\frac{b}{4}(u_2+u_3)(v_2+v_3)-\frac{c}{4}(u_2-u_3)(v_2-v_3)-au_4v_4=0
\endaligned
\right.
$$
over the~base $T=\mathrm{Spec}(\mathbb{C}[a,b,c,d])$.
This family is invariant under the~$\mathbb{C}^*$-action with weights $(-1, 0,0,0)$ on $(u_1,u_2,u_3,u_4)$,
with weights $(-1, 0,0,0)$ on $(v_1,v_2,v_3,v_4)$, and with weights $(0,0,0,1)$ on the~coordinates of the~base $(a,b,c,d)$.
The invariant base of the~family under this $\mathbb{C}^*$-action is $\mathrm{Spec}(T^{\mathbb{C}^*})\simeq\mathbb{C}^3$.
Let
\begin{align*}
Z^{+}&=\mathbb{C}^{4}_{a,b,c,d}\setminus\{d=0\},\\
Z^{-}&=\mathbb{C}^{4}_{a,b,c,d}\setminus\{a=b=c=0\}.
\end{align*}
Then there are morphisms $\alpha\colon Z^{+}/\mathbb{C}^*\to\mathrm{Spec}(T^{\mathbb{C}^*})$ and $\beta\colon Z^{-}/\mathbb C^*\to\mathrm{Spec}(T^{\mathbb{C}^*})$ such~that
\begin{itemize}
\item $\alpha$ is an~isomorphism and gives an~affine neighborhood of $(0:0:0:1)\in\mathbb{P}^3_{a,b,c,d}$,
\item $\beta$ is a~blow-up of the~point $(0:0:0:1)$.
\end{itemize}
\end{remark}

Therefore, blowing up $\mathbb{B}\to\mathbb{P}^3_{a,b,c,d}$ at the~$W(\mathrm{F}_4)/\langle -1\rangle$-orbit of the~point $(0:0:0:1)$,
we obtain a~new deformation family
$$
\mathcal{X}\to \mathbb{B}
$$
that generically parametrises smooth Fano 3-folds in the~family~\textnumero 4.1.

Let $\mathbb{E}$ be the~exceptional divisor of this blow up over $(0:0:0:1)$.
Then $\mathbb{E}$ is canonically isomorphic to $\mathbb{P}^2_{a,b,c}$,
and closed points of $\mathbb{E}$ parametrise singular divisors of degree $(2,2)$ in the~product $\mathbb{P}(1_{s_1},1_{t_1},2_{w_1})\times\mathbb{P}(1_{s_2},1_{t_2},2_{w_2})$ that are given by \eqref{equation:form-2}.

\begin{remark}
\label{remark:1-2-3}
If $\Theta$ is the~stabilizer in $W(\mathrm{F}_4)/\langle -1\rangle$ of $(0:0:0:1)$,
then
$$
\Theta\simeq\mathfrak{S}_4\times(\mathbb{Z}/2\mathbb{Z})
$$
the group $\Theta$ is isomorphic to the~Weyl group $W(\mathrm{B}_3)$,
and it acts on $\mathbb{E}$ with kernel $(\mathbb{Z}/2\mathbb{Z})$,
so that its image in $\mathrm{Aut}(\mathbb{E})$ is isomorphic to $\mathfrak{S}_4$.
This gives $\mathbb{E}/\Theta\simeq\mathbb{P}(1,2,3)$.
\end{remark}

Now, let $X_{(a:b:c)}$ be the~hypersurface in $\mathbb{P}(1_{s_1},1_{t_1},2_{w_1})\times\mathbb{P}(1_{s_2},1_{t_2},2_{w_2})$ given by~\eqref{equation:form-2}~for some \mbox{$(a:b:c)\in\mathbb{P}^2$},
and let $G$ be the~subgroup of $\mathrm{Aut}(X)$ generated by the~involutions
\begin{align*}
\tau_{1}\colon\big((s_1:t_1:w_1),(s_2:t_2:w_2)\big)&\mapsto\big((s_2:t_2:w_1),(s_1:t_1:w_2)\big),\\
\tau_{2}\colon\big((s_1:t_1:w_1),(s_2:t_2:w_2)\big)&\mapsto\big((t_1:s_1:w_1),(t_2:s_2:w_2)\big),\\
\tau_{3}\colon\big((s_1:t_1:w_1),(s_2:t_2:w_2)\big)&\mapsto\big((s_1:-t_1:w_1),(s_2:-t_2:w_2)\big),
\end{align*}
and by the~automorphisms
$$
\sigma_{\lambda}:\big((s_1:t_1:w_1),(s_2:t_2:w_2)\big)\mapsto\Bigg(\Big(s_1:t_1:\lambda w_1\Big),\Big(s_2:t_2:\frac{w_2}{\lambda}\Big)\Bigg),
$$
for $\lambda\in\mathbb{C}^\ast$.
Then $G\simeq\mathbb{C}^\ast\rtimes(\mathbb{Z}/2\mathbb{Z})^3$.
Observe that $X$ is singular along the~curves
\begin{align*}
C_1&=\big\{s_1=t_1=w_2=0\big\},\\
C_2&=\big\{s_2=t_2=w_1=0\big\}.
\end{align*}
To describe all possible singularities of $X_{(a:b:c)}$, we set
\begin{align*}
\mathbb{E}^{\operatorname{Sing}+}&=\left\{(a:b:c)\in\mathbb{E}\ |\ \operatorname{Sing}\big(X_{(a:b:c)}\big)\ne C_1\cup C_2\right\},\\
\mathbb{E}^{\operatorname{Curv}+}&=\left\{(a:b:c)\in\mathbb{E}\ |\ \operatorname{Sing}\big(X_{(a:b:c)}\big)\setminus (C_1\cup C_2) \text{ is a~curve}\right\},\\
\mathbb{E}^{4\mathbb{A}_1+}&=\left\{(a:b:c)\in\mathbb{E}\ |\ X_{(a:b:c)} \text{ has $4$ isolated $\mathbb{A}_1$ singularities outside } C_1\cup C_2\right\},\\
\mathbb{E}^{2\mathbb{A}_1+}&=\left\{(a:b:c)\in\mathbb{E}\ |\ X_{(a:b:c)} \text{ has $2$ isolated $\mathbb{A}_1$ singularities outside } C_1\cup C_2\right\}.
\end{align*}
Here and in the~following, we identify $\mathbb{E}=\mathbb{P}^2_{a,b,c}$.

\begin{proposition}
\label{proposition:stratify2nd}
We have the~stratification
$$
\mathbb{E}^{\operatorname{Sing}+}=
\mathbb{E}^{\operatorname{Curv}+}\sqcup
\mathbb{E}^{4\mathbb{A}_1+}\sqcup
\mathbb{E}^{2\mathbb{A}_1+},
$$
satisfying that
$$
\overline{\mathbb{E}^{2\mathbb{A}_1+}}=\mathbb{E}^{\operatorname{Sing}+}, \quad \mathbb{E}^{\operatorname{Sing}+}\setminus\mathbb{E}^{2\mathbb{A}_1+}=\mathbb{E}^{\operatorname{Curv}+}\sqcup\mathbb{E}^{4\mathbb{A}_1+}\sqcup\mathbb{E}^{2\mathbb{A}_1+}.
$$
Moreover, one has
\begin{align*}
\mathbb{E}^{\operatorname{Sing}+}&=\big\{\left({a}^2-{b}^2\right)\left({a}^2-{c}^2\right)\left({b}^2-{c}^2\right)=0\big\},\\
\mathbb{E}^{\operatorname{Curv}+}&=\{(1:1:1),(1:1:-1),(1:-1:-1),(1:-1:1)\big\},\\
\mathbb{E}^{4\mathbb{A}_1+}&=\{(1:0:0),(0:1:0),(0:0:1)\big\}.
\end{align*}
Furthermore, the~subsets $\mathbb{E}^{\operatorname{Curv}+}$, $\mathbb{E}^{4\mathbb{A}_1+}$, $\overline{\mathbb{E}^{2\mathbb{A}_1+}}$
are cut out on the~divisor $\mathbb{E}$ by the~strict transforms of the~subsets $\mathbb{P}^{\operatorname{Curv}}$,
$\overline{\mathbb{P}^{4\mathbb{A}_1}}$, $\overline{\mathbb{P}^{2\mathbb{A}_1}}$ defined in Section~\ref{section:geometry}, respectively.
\end{proposition}

\begin{proof}
Direct computations.
\end{proof}

Let us describe $\operatorname{Sing}(X)\setminus(C_1\cup C_2)$ explicitly.
If $(a:b:c)\in\mathbb{E}^{\operatorname{Curv}+}$, then
$$
\operatorname{Sing}(X)\setminus(C_1\cup C_2)=\left\{\aligned
&\{w_1=w_2=s_1s_2+t_1t_2=0\} \text{ if } (a:b:c)=(1:1:1), \\
&\{w_1=w_2=s_1t_2+t_1s_2=0\} \text{ if } (a:b:c)=(1:1:-1), \\
&\{w_1=w_2=s_1s_2-t_1t_2=0\} \text{ if } (a:b:c)=(1:-1:-1), \\
&\{w_1=w_2=s_1t_2-t_1s_2=0\} \text{ if } (a:b:c)=(1:-1:1),
\endaligned
\right.
$$
and $X_{(a:b:c)}$ has $\mathbb{A}_1$ singularity along this curve.
Moreover, if $(a:b:c)=(1:0:0)$, then
\begin{align*}
\operatorname{Sing}(X)\setminus(C_1\cup C_2)=\Big\{&\big((1:0:0),(1:0:0)\big),(1:0:0),(0:1:0)\big),\\
&\big((0:1:0),(1:0:0)\big),\big((0:1:0),(0:1:0)\big)\Big\}.
\end{align*}
If $(a\colon b\colon c)=(0:1:0)$, then
\begin{align*}
\operatorname{Sing}(X)\setminus(C_1\cup C_2)=\Big\{&\big((i:1:0),(i:1:0)\big), \big((i:1:0),(-i:1:0)\big),\\
&\big((-i:1:0),(i:1:0)\big),\big((-i:1:0),(-i:1:0)\big)\Big\}.
\end{align*}
If $(a\colon b\colon c)=(0:0:1)$, then
\begin{align*}
\operatorname{Sing}(X)\setminus(C_1\cup C_2)=\Big\{&\big((1:1:0),(1:1:0)\big),\big((1:1:0),(1:-1:0)\big),\\
&\big((1:-1:0),(1:1:0)\big), \big((1:-1:0),(1:-1:0)\big)\Big\}.
\end{align*}
Note that
$X_{(1:0:0)}$, $X_{(0:1:0)}$, $X_{(0:0:1)}$ are isomorphic to the~K-polystable toric Fano 3-fold with reflexive ID 610 in the~Graded Rings Database \cite{Kasprzyk}.
If $(a:b:c)\in\mathbb{E}^{2\mathbb{A}_1+}$, then
$$
\operatorname{Sing}(X)\setminus(C_1\cup C_2)=\left\{\aligned
&\big((1:0:0),(1:0:0)\big)\cup\big((0:1:0),(0:1:0)\big)  \text{ if } b+c=0, \\
&\big((1:0:0),(0:1:0)\big)\cup\big((0:1:0),(1:0:0)\big)  \text{ if } b-c=0, \\
&\big((1:1:0),(1:-1:0)\big)\cup\big((1:-1:0),(1:1:0)\big)  \text{ if } a-b=0, \\
&\big((1:i:0),(1:i:0)\big)\cup\big((1:-i:0),(1:-i:0)\big) \text{ if } a-c=0, \\
&\big((1:i:0),(1:-i:0)\big)\cup\big((1:-i:0),(1:i:0)\big)  \text{ if } a+c=0, \\
&\big((1:1:0),(1:1:0)\big)\cup\big((1:-1:0),(1:-1:0)\big)  \text{ if } a+b=0.
\endaligned
\right.
$$
In the~next section we will show that $X_{(a:b:c)}$ is K-polystable for every $(a:b:c)\in\mathbb{P}^2$.

\section{K-polystable degenerations}
\label{section:K-stability}

Since all smooth elements of family \textnumero 4.1 are known to be K-polystable \cite{BelousovLoginov},
we now investigate singular K-polystable limits of elements in the~family.

\subsection{Case \eqref{equation:form-1}}
\label{subsection:K-polystable-1-1-1-1}
Let $X$ be an~divisor in  $\mathbb{P}^1_{x_1,y_1}\times\mathbb{P}^1_{x_2,y_2}\times\mathbb{P}^1_{x_3,y_3}\times\mathbb{P}^1_{x_4,y_4}$
that is given by zeroes of the~form \eqref{form:1} for some $(a,b,c,d)\ne(0,0,0,0)$. Suppose that $X$ is irreducible.

\begin{theorem}
\label{theorem:main-1}
The Fano 3-fold $X$ is K-polystable.
\end{theorem}

For each $i\in\{1,2,3,4\}$, let $\operatorname{pr}^i \colon X \to (\mathbb{P}^1)^3$ be the~projection to all but the~$i^{\mathrm{th}}$ factor.
Then each $\operatorname{pr}^i$ is a~birational morphism (if $X$ is smooth, then each $\operatorname{pr}^i$ blows up a~smooth elliptic curve).
Let $E_i$ be $\operatorname{pr}^i$-exceptional surface.

\begin{lemma}
\label{lemma:singular-locus}
One has $\operatorname{Sing}(X)=E_1\cap E_2\cap E_3\cap E_4$
\end{lemma}

\begin{proof}
Write the~equation of $X$ as $x_i f_i+ y_i g_i=0$
for forms $f_i$ and $g_i$ in $H^0 ((\mathbb{P}^1)^3, \mathcal{O}(1,1,1))$.
Then $E_i=\{f_i=g_i=0\}$. Now, apply the~Jacobian criterion.
\end{proof}

Let us use notations of Section~\ref{section:geometry}.
Recall that $\mathrm{Aut}(X)$ contains the~subgroup
$$
G=\langle \tau_1,\tau_2,\tau_3,\tau_4\rangle\simeq\big(\mathbb{Z}/2\mathbb{Z}\big)^4.
$$

\begin{lemma}
\label{lemma:K-stability-nodal-easy}
Suppose that $X$ has isolated singularities. Then $X$ is K-polystable.
\end{lemma}

\begin{proof}
We suppose that  $X$ is not K-polystable.
Then it follows from \cite{Fujita2019,Li,Zhuang} that there exists a~$G$-invariant prime divisor $\mathbf{F}$ over $X$ such that $\beta(\mathbf{F})\leqslant 0$.
Let $Z$ be its~center~on~$X$.
Then $Z$ is not a~point, because $G$ does not fix~points in $X$.
Thus, either $Z$ is an~irreducible curve or an~irreducible surface.
Hence, since $X$ has isolated singularities by assumption, it follows from Lemma~\ref{lemma:singular-locus}
that $Z$ is not contained in $E_1\cup E_2\cup E_3\cup E_4$.

Note also that $Z$ is mapped surjectively to $\mathbb{P}^1$ via the~projection $X\to\mathbb{P}^1$ to any factor,
because the~subgroup $\langle\tau_3,\tau_4\rangle\subset G$ does not fix points in $\mathbb{P}^1$.

Let $P$ be a~general point in $Z$, and let $S$ be the~fiber of the~projection $X\to\mathbb{P}^1$ to the~first factor of $(\mathbb{P}^1)^4$ that contains $P$.
Then $P\not\in E_1\cup E_2\cup E_3\cup E_4$,
and $S$ is a~smooth sextic del Pezzo surface.
Fix $u\in\mathbb{R}_{\geqslant 0}$. Then $-K_X-uS$ is pseudo-effective $\iff$ $u\leqslant 2$.
For~$u\in[0,2]$, let $P(u)$ be the~positive part of the~Zariski decomposition of $-K_X-uS$,
and let $N(u)$ be its negative part. Note that this decomposition exists on $X$:  we have
$$
P(u)\sim_{\mathbb{R}} \left\{\aligned
&-K_X-uS \ \text{ for } 0\leqslant u\leqslant 1, \\
&-K_X-uS-(u-1)E_1\ \text{ for } 1\leqslant u\leqslant 2,
\endaligned
\right.
$$
and
$$
N(u)= \left\{\aligned
&0\ \text{ for } 0\leqslant u\leqslant 1, \\
&(u-1)E_1\ \text{ for } 1\leqslant u\leqslant 2.
\endaligned
\right.
$$
Integrating, we get $S_X(S)=\frac{1}{24}\int\limits_{0}^{2}\big(P(u)^3\big)du=\frac{11}{16}$, so that $\beta(S)=1-S_X(S)=\frac{5}{16}>0$.
Thus, since $\delta(S)=1$ \cite{Book}, applying \cite[Lemma~26]{CheltsovFujitaKishimotoOkada}, known as the~Nemuro Lemma, we get
$$
\delta_P(X)\geqslant \frac{(-K_X)^3}{3}\frac{\delta(S)}{2(-K_S)^2}=\frac{2}{3}.
$$
But the~proof of the~Nemuro Lemma gives $\delta_P(X)\geqslant\frac{16}{15}$.
Indeed, recall from \cite{AbbanZhuang,Book} that
$$
\delta_P(X)\geqslant\min\Bigg\{\frac{1}{S_X(S)},\inf_{\substack{F/S\\P\in C_S(F)}}\frac{A_S(F)}{S\big(W^S_{\bullet,\bullet};F\big)}\Bigg\},
$$
where the~infimum is taken by all prime divisors over $S$ whose center on $S$ contains $P$,~and
$$
S\big(W^S_{\bullet,\bullet};F\big)=\frac{3}{(-K_X)^3}\int\limits_0^{2}\Big(P(u)\big\vert_{S}\Big)^2\mathrm{ord}_F(N(u))du+
\frac{3}{(-K_X)^3}\int\limits_0^2\int\limits_0^{\infty}\mathrm{vol}\big(P(u)\big\vert_{S}-vF\big)dvdu.
$$
We claim that $S(W^S_{\bullet,\bullet};F)\leqslant\frac{15}{16}A_S(F)$ for every prime divisor over the~surface $S$ whose center on $S$ contains $P$.
Indeed, since $P\not\in E_1$, we have
\begin{multline*}
\quad\quad\quad\quad\quad\quad\quad\quad S\big(W^S_{\bullet,\bullet};F\big)=\frac{3}{(-K_X)^3}\int\limits_0^2\int\limits_0^{\infty}\mathrm{vol}\big(P(u)\big\vert_{S}-vF\big)dvdu=\\
=\frac{1}{8}\int\limits_0^1\int\limits_0^{\infty}\mathrm{vol}\big(-K_S-vF\big)dvdu+
\frac{1}{8}\int\limits_1^2\int\limits_0^{\infty}\mathrm{vol}\big((2-u)(-K_S)-vF\big)dvdu=\\
=\frac{1}{8}\int\limits_0^1\int\limits_0^{\infty}\mathrm{vol}\big(-K_S-vF\big)dvdu+
\frac{1}{8}\int\limits_1^2(2-u)^3\int\limits_0^{\infty}\mathrm{vol}\big(-K_S-vF\big)dvdu.
\end{multline*}
But it follows from the~definition of $\delta(S)$ that
$$
\frac{1}{(-K_S)^2}\int\limits_0^{\infty}\mathrm{vol}\big(-K_S-vF\big)\leqslant \frac{A_S(F)}{\delta(S)}.
$$
Thus, we conclude that
\begin{multline*}
\quad\quad\quad\quad\quad S\big(W^S_{\bullet,\bullet};F\big)\leqslant \frac{1}{8}\int\limits_0^1\frac{(-K_S)^2}{\delta(S)}A_S(F)du+\frac{1}{8}\int\limits_1^2(2-u)^3\frac{(-K_S)^2}{\delta(S)}A_S(F)du=\\
=\frac{3}{4}A_S(F)+\frac{3}{4}A_S(F)\int\limits_1^2(2-u)^3du=\frac{15}{16}A_S(F).\quad\quad\quad\quad\quad
\end{multline*}
This gives $\delta_P(X)\geqslant\frac{16}{15}$ as claimed.
On the~other hand, we have $\delta_P(X)\leqslant 1$, since $\beta(\mathbf{F})\leqslant 0$.
The obtained contradiction completes the~proof of the~lemma.
\end{proof}

Hence, to complete the~proof of Theorem~\ref{theorem:main-1}, we may assume that $\mathrm{dim}(\mathrm{Sing}(X))=1$.

Recall from Section~\ref{section:geometry} that $\operatorname{Sing}(X)$ consists of a~single curve,
and $X$ has ordinary double singularity along this curve.
Let $\eta\colon\widetilde{X}\to X$ be the~blow-up of the~curve $\operatorname{Sing}(X)$,
let $F$ the~exceptional divisor of the~blow-up $\eta$,
and let $\widetilde{E}_1$, $\widetilde{E}_2$, $\widetilde{E}_3$, $\widetilde{E}_4$ be the~proper transforms on the~3-fold $\widetilde{X}$ of
the surfaces $E_1$, $E_2$, $E_3$, $E_4$, respectively.
Then it follows from Appendix~\ref{section:4-lines} that we have the~following $G$-equivariant commutative diagram:
$$
\xymatrix{
&\widetilde{X}\ar@{->}[dl]_{\pi}\ar@{->}[dr]^{\eta}\\%
\mathbb{P}^3\ar@{-->}[rr]_{\chi} && X}
$$
where $\pi$ is a~birational morphism contracting $\widetilde{E}_1, \widetilde{E}_2,\widetilde{E}_3$, $\widetilde{E}_4$
to skew lines $L_1$, $L_2$, $L_3$, $L_4$, respectively, and $\pi(F)$ is a~smooth quadric surface containing those lines.
The rational map $\chi$ is the~product of the~linear projections $\mathbb{P}^3\dasharrow\mathbb{P}^1$ from $L_1$, $L_2$, $L_3$, $L_4$.
Note that
$$
\eta^*(-K_X)\sim -K_{\widetilde{X}}\sim 2F+\widetilde{E}_1+\widetilde{E}_2+\widetilde{E}_3+\widetilde{E}_4.
$$
From this diagram, we immediately see that $\mathrm{Aut}^0(X)\simeq\mathrm{PGL}_2(\mathbb{C})$,
and $\mathrm{Aut}(X)$ contains a~subgroup isomorphic to $\mathrm{PGL}_2(\mathbb{C})\times(\mathbb{Z}/2\mathbb{Z})^2$
such that the~factor $(\mathbb{Z}/2\mathbb{Z})^2$ permutes transitively the~surfaces $\widetilde{E}_1, \widetilde{E}_2,\widetilde{E}_3$, $\widetilde{E}_4$.
Then $F$ is the~only $\mathrm{Aut}(X)$-invariant prime divisor over $X$,
so it is enough to show that $\beta(F)>0$ to prove that $X$ is K-polystable~\cite{Zhuang}.

Fix $u\in\mathbb{R}_{\geqslant 0}$, $-K_X-uF$ is pseudo-effective if and only if $u\leqslant 2$, since
$$
\eta^*(-K_X)-uF\sim_{\mathbb{R}} (2-u)F+\widetilde{E}_1+\widetilde{E}_2+\widetilde{E}_3+\widetilde{E}_4.
$$
For~$u\in[0,2]$, let $P(u)$ be the~positive part of the~Zariski decomposition of $-K_X-uF$,
and let $N(u)$ be its negative part. Note that this decomposition exists on $X$:  we have
$$
P(u)\sim_{\mathbb{R}} \left\{\aligned
&(2-u)F+\widetilde{E}_1+\widetilde{E}_2+\widetilde{E}_3+\widetilde{E}_4 \ \text{ for } 0\leqslant u\leqslant 1, \\
&(2-u)\big(F+\widetilde{E}_1+\widetilde{E}_2+\widetilde{E}_3+\widetilde{E}_4\big)\ \text{ for } 1\leqslant u\leqslant 2,
\endaligned
\right.
$$
and
$$
N(u)= \left\{\aligned
&0\ \text{ for } 0\leqslant u\leqslant 1, \\
&(u-1)\big(\widetilde{E}_1+\widetilde{E}_2+\widetilde{E}_3+\widetilde{E}_4\big)\ \text{ for } 1\leqslant u\leqslant 2.
\endaligned
\right.
$$
This gives $S_X(F)=\frac{1}{24}\int\limits_{0}^{2}P(u)^3du=\frac{1}{24}\int\limits_{0}^{1}8u^3-24u^2+24du+\frac{1}{24}\int\limits_{1}^{2}8(2-u)^3udu=\frac{5}{6}$,
which implies that $\beta(F)=1-S_X(F)=\frac{1}{6}>0$. This shows that $X$ is K-polystable.

\subsection{Case \eqref{equation:form-2}}
\label{subsection:form-2}
Let $X$ be a~divisor in $\mathbb{P}(1_{s_1},1_{t_1},2_{w_1})\times \mathbb{P}(1_{s_2},1_{t_2},2_{w_2})$ given by \eqref{equation:form-2}.

\begin{theorem}
\label{theorem:main-2}
The Fano 3-fold $X$ is K-polystable.
\end{theorem}

Let us use notations of Section~\ref{section:2-2-divisors}.
Recall that $\mathrm{Aut}(X)$ has subgroup $G\simeq\mathbb{C}^\ast\rtimes(\mathbb{Z}/2\mathbb{Z})^3$,
which is generated by the~involutions $\tau_1$, $\tau_2$, $\tau_3$, and automorphisms $\sigma_\lambda$, where $\lambda\in\mathbb{C}^\ast$.
Let $\rho\colon X\dasharrow \mathbb{P}^1\times\mathbb{P}^1$ be the~rational map
$$
\big((s_1:t_1:w_1),(s_2:t_2:w_2)\big)\mapsto\big((s_1:t_1),(s_2:t_2)\big),
$$
let $\eta\colon\widetilde{X}\to X$ be the~blow-up of the~curves $C_1$ and $C_2$,
and let $E_1$ and $E_2$ be the~exceptional surfaces of $\eta$ that are mapped $C_1$ and $C_2$, respectively.
Then $\eta$ an~$\rho$ are $G$-equivariant,
and we have the~following $G$-equivariant commutative diagram:
\begin{equation}
\label{CBdiag}
\xymatrix{
&\widetilde{X}\ar@{->}[dr]^{\pi}\ar@{->}[dl]_{\eta}\\%
X\ar@{-->}[rr]_{\rho} &&\mathbb{P}^1\times\mathbb{P}^1}
\end{equation}
where $\pi$ is a~(non-standard) conic bundle, of which $E_1$ and $E_2$ are sections.

\begin{lemma}
\label{lemma:K-stability-PGL2-again}
If $\widetilde{X}$ has non-isolated singularities, $X$ is K-polystable.
\end{lemma}

\begin{proof}
We may assume that $(a:b:c)=(1:-1:1)$. Then $\operatorname{Sing}(X)=C_1\cup C_2\cup C^\prime$,~where
$$
C^\prime=\big\{s_1t_2-s_2t_1=w_1=w_2=0\big\}.
$$
Observe that $C^\prime=F_1\cap F_2$, where
\begin{align*}
F_1&=\{ s_1t_2-s_2t_1=w_1=0\},\\
F_2&=\{s_1t_2-s_2t_1=w_2=0\}.
\end{align*}
Let $\widetilde{F}_1$ and $\widetilde{F}_2$ be the~strict transforms of $F_1$ and $F_2$ and $\widetilde{C}^\prime=\widetilde{F}_1\cap \widetilde{F}_2$ their intersection. Under $\pi$, the~image of $\widetilde{C}^\prime$ is a~smooth $(1,1)$-curve and $\pi^*(\pi (\widetilde{C}^\prime))= \widetilde F_1+ \widetilde F_2$.

Let $\phi\colon\widehat{X}\to \widetilde{X}$ be the~blow-up of the~curve $\widetilde{C}^\prime$,
let $E^\prime$ be the~$\phi$-exceptional surface, and let $\widehat{E}_1$, $\widehat{E}_2$, $\widehat{F}_1$, $\widehat{F}_2$
be the~strict transforms on $\widehat{X}$ of the~surfaces $E_1$, $E_2$, $F_1$, $F_2$, respectively.
Then it follows from Appendix~\ref{section:4-lines} that there is $G$-equivariant commutative diagram
$$
\xymatrix{
&&\widehat{X}\ar@{->}[dr]^{\phi}\ar@{->}[dl]_{\psi}&&\\%
&\overline{X}\ar@{->}[dr]^{\varpi}\ar@{->}[dl]_{\theta}&&\widetilde{X}\ar@{->}[dr]^{\eta}\ar@{->}[dl]_{\pi}&\\%
\mathbb{P}^3\ar@{-->}[rr]_{\chi} &&\mathbb{P}^1\times\mathbb{P}^1&&X\ar@{-->}[ll]^{\rho}}
$$
where: $\psi$ is a~birational morphism that contracts $\widehat{F}_1$ and $\widehat{F}_2$,
$\theta$ is a~birational morphism that contracts $\psi(\widehat{E}_1)$ and $\psi(\widehat{E}_2)$ to two skew lines,
$\chi$ is the~product of the~linear projections from $\psi(\widehat{E}_1)$ and $\psi(\widehat{E}_2)$,
and $\varpi$ is a~$\mathbb{P}^1$-bundle.
Then the~surface $\theta\circ\psi(E^\prime)$ is a~smooth quadric that contains the~lines $\psi(\widehat{E}_1)$ and $\psi(\widehat{E}_2)$.
Thus, we have
$$
(\eta\circ\phi)^*(-K_X)\sim -K_{\widehat{X}}\sim 2E^\prime+\widehat{E}_1+\widehat{E}_2+2\widehat{F}_1+2\widehat{F}_2.
$$
The map $\psi$ is the~blow-up of the~curves $\psi(E^\prime)\cap\psi(\widehat{E}_1)$ and $\psi(E^\prime)\cap\psi(\widehat{E}_2)$.
Then
$$
\mathrm{Aut}(X)\cong\mathrm{PGL}_2\big(\mathbb{C}\big)\times\big(\mathbb{C}^\ast\rtimes(\mathbb{Z}/2\mathbb{Z})\big),
$$
and the~surface $E^\prime$ is the~only \mbox{$\mathrm{Aut}(X)$-invariant} prime divisor over $X$.
Hence, it is enough to check that $\beta(E^\prime)>0$ to show that $X$ is K-polystable \cite{Zhuang}.

Fix $u\in\mathbb{R}_{\geqslant 0}$. Then $-K_X-uE^\prime$ is pseudo-effective if and only if $u\leqslant 2$, since
$$
(\eta\circ\phi)^*(-K_X)-uE^\prime\sim_{\mathbb{R}} (2-u)E^\prime+\widehat{E}_1+\widehat{E}_2+2\widehat{F}_1+2\widehat{F}_2.
$$
For~$u\in[0,2]$, let $P(u)$ be the~positive part of the~Zariski decomposition of $-K_X-uE^\prime$,
and let $N(u)$ be its negative part. Note that this decomposition exists on $X$: we have
$$
P(u)\sim_{\mathbb{R}} \left\{\aligned
(2-u)E^\prime+\widehat{E}_1+\widehat{E}_2+2\widehat{F}_1+2\widehat{F}_2 \ \text{ for } 0\leqslant u\leqslant 1, \\
(2-u)\big(E^\prime+\widehat{E}_1+\widehat{E}_2+2\widehat{F}_1+2\widehat{F}_2\big)\ \text{ for } 1\leqslant u\leqslant 2,
\endaligned
\right.
$$
and
$$
N(u)= \left\{\aligned
&0\ \text{ for } 0\leqslant u\leqslant 1, \\
&(u-1)\big(\widehat{E}_1+\widehat{E}_2+2\widehat{F}_1+2\widehat{F}_2\big)\ \text{ for } 1\leqslant u\leqslant 2.
\endaligned
\right.
$$
This gives $S_X(E^\prime)=\frac{1}{24}\int\limits_{0}^{2}\big(P(u)\big)^3du=
\frac{1}{24}\int\limits_{0}^{1}8u^3-24u^2+24du+\frac{1}{24}\int\limits_{1}^{2}8(2-u)^3udu=\frac{5}{6}$,
which implies  that $\beta(E^\prime)=1-S_X(E^\prime)=\frac{1}{6}>0$. This shows that $X$ is K-polystable.
\end{proof}

Hence, to prove Theorem~\ref{theorem:main-2},
we may assume that $\widetilde{X}$ has isolated singularities:
$$
(a:b:c)\neq (1:\pm 1: \pm 1).
$$
Then one of the~following cases holds:
\begin{itemize}
\item $(a:b:c)\not\in\mathbb{E}^{\operatorname{Sing}+}$, and $\widetilde{X}$ is smooth;
\item $(a:b:c)\in\mathbb{E}^{2\mathbb{A}_1+}$, and $\widetilde{X}$ has $2$ ordinary double points;
\item $(a:b:c)\in\mathbb{E}^{4\mathbb{A}_1+}$, and $X$ is toric.
\end{itemize}
If $X$ is toric, it is K-polystable (see Section~\ref{section:2-2-divisors}).
Hence, we may assume that $\widetilde{X}$ is smooth, or  $\widetilde{X}$ has 2 ordinary double points. Set
\begin{eqnarray*}
F_1=\left\{as_1t_1s_2t_2 +\frac{b+c}{4}(s_1^2s_2^2+t_1^2t_2^2)+\frac{b-c}{4}(s_1^2t^2_2+t^2_1s^2_2)=w_1=0\right\} \\
F_2=\left\{as_1t_1s_2t_2 +\frac{b+c}{4}(s_1^2s_2^2+t_1^2t_2^2)+\frac{b-c}{4}(s_1^2t^2_2+t^2_1s^2_2)=w_2=0\right\}.
\end{eqnarray*}
Denote by $\widetilde{F}_1$ and $\widetilde{F}_1$ the~strict transforms of $F_1$ and $F_2$ on $\widetilde{X}$, respectively.~Then
$$
\eta^*\big(-K_X\big)\sim-K_{\widetilde{X}}\sim \widetilde{F}_1+\widetilde{F}_2+E_1+E_2.
$$
If $\widetilde{X}$ is smooth, then $F_1$ and $F_2$ are irreducible.
If $\widetilde{X}$ has $2$ ordinary double points, then these surfaces have two components.

\begin{lemma}
\label{lemma:class-group}
Let $S$ be a~$G$-invariant irreducible surface in $X$.
Then
$$
S\sim_{\mathbb{Q}}\frac{n}{2}(-K_X)
$$
for some $n\in\mathbb{N}$.
\end{lemma}

\begin{proof}
Left to the~reader.
\end{proof}

\begin{lemma}
\label{lemma:G-invariant-curves}
If $\widetilde{X}$ is smooth, then $F_1\cap F_2$ is the~only $G$-invariant irreducible curve~in~$X$.
Similarly, if $\widetilde{X}$ is singular, then $X$ does not contain $G$-invariant irreducible curves.
\end{lemma}

\begin{proof}
If $Z$ is a~$G$-invariant irreducible curve in $X$,
ut is pointwise fixed by the~\mbox{$\mathbb{C}^\ast$-action}, because $G$ does not fix points in $\mathbb{P}^1\times\mathbb{P}^1$.
Then $Z=F_1\cap F_2$, so $\widetilde{X}$ is smooth, since otherwise the~intersection $F_1\cap F_2$ would be reducible.
\end{proof}

Now, we are ready to prove that $X$ is K-polystable.
Suppose that  $X$ is not K-polystable.
Then, by \cite{Fujita2019,Li,Zhuang}, there exists a~$G$-invariant prime divisor $\mathbf{F}$ over $X$ such that $\beta(\mathbf{F})\leqslant 0$.
Let $Z$ be the~center of this divisor on $X$.
Then $Z$ is not a~point, since $G$ does not fix~points.
Similarly, we see that $Z$ is not a~surface by Lemma~\ref{lemma:class-group}.
Thus, we see that $Z$ is a~curve.
Then $\widetilde{X}$ is smooth, and $Z=F_1\cap F_2$ by Lemma~\ref{lemma:G-invariant-curves}.

Let $\widetilde{Z}$ be the~strict transform on $\widetilde{X}$ of the~curve $Z$,
let $\xi\colon\widehat{X}\to \widetilde{X}$ be the~ordinary blow-up along $\widetilde{Z}$,
and let $E$ be the~exceptional divisor of $\xi$.
Then the~$G$-action lifts to $\widehat{X}$,
and $E$ does not contains $G$-invariant irreducible curves.
Hence, $E$ is the~only $G$-invariant prime divisor over $X$.
So, it follows from \cite{Zhuang} that $\beta(E)\leqslant 0$.
Let us compute $\beta(E)$.

Fix $u\in\mathbb{R}_{\geqslant 0}$.
If $u\in[0,1]$, then $\xi^*(-K_X)-uE$ is nef.
If $u\in[1,2]$, then the~positive part of the~Zariski decomposition of $\xi^*(-K_X)-uE$ is
$$
(u-1)\left(\widehat{E}_1+\widehat{E}_2+\widehat{F}_1+\widehat{F}_2\right),
$$
where $\widehat{E}_1$, $\widehat{E}_2$, $\widehat{F}_1$, $\widehat{F}_2$ are the~strict transforms of the~surfaces $E_1$, $E_2$, $F_1$, $F_2$, respectively.
This gives
$$
\operatorname{vol}\left(\xi^*(-K_X)-u E\right)=\begin{cases}
8(3-3u^2+u^3) & \text{if } u\in[0,1], \\
8(2-u)^3 & \text{if } u\in[0,2].
\end{cases}
$$
Thus, integrating, we get $S_X(E)=\frac{5}{6}<A_X(E)=2$, so $\beta(E)>0$ --- a~contradiction.

This finishes the~proof of Theorem~\ref{theorem:main-2}.

\section{K-moduli space}
\label{section:K-moduli}

To describe the~scheme-structure of the~K-moduli space,
we need the~following result:

\begin{proposition}
\label{proposition:deform}
Let $X$ be a~complete intersection in $\mathbb{P}^3\times\mathbb{P}^3$ of three divisors that have degrees $(2,0)$, $(0,2)$ and $(1,1)$.
Suppose that $X$ has canonical Gorenstein singularities.
Then the~deformations of $X$ are unobstructed.
\end{proposition}

\begin{proof}
The proof is almost same as the~proof of
\cite[Proposition 2.9]{CheltsovDuarteFujitaKrylovMartinez}.
Namely, it follows from
\cite[Theorem~2.3.2 and~2.4.1]{Sernesi},
\cite[Corollary 2.4.2]{Sernesi},
\cite[Propositions 2.4 and 2.6]{Sano} that
it is enough to show that
$\operatorname{Ext}^2_{\mathcal{O}_X}\left(\Omega_X^1,\mathcal{O}_X\right)=0$.
As in \cite[\S 1.2]{Sano}, we know that
$$
\operatorname{Ext}^2_{\mathcal{O}_X}\left(\Omega_X^1,\mathcal{O}_X\right)
\simeq H^1\left(X, \Omega_X^1\otimes \omega_X\right)^\vee.
$$
From the~exact sequence
$$
0\to\mathcal{N}^\vee_{X/\mathbb{P}^3\times\mathbb{P}^3}\to
\Omega^1_{\mathbb{P}^3\times\mathbb{P}^3}|_X\to\Omega_X^1\to 0,
$$
we have
$$
H^1\left(X, \Omega^1_{\mathbb{P}^3\times\mathbb{P}^3}\big|_X\otimes\omega_X\right)\simeq H^1\left(X, \Omega_X^1\otimes \omega_X\right).
$$
Indeed, for $i\in\{1,2\}$, we have
\begin{multline*}
H^i\left(X, \mathcal{N}^\vee_{X/\mathbb{P}^3\times\mathbb{P}^3}\otimes\omega_X\right)\simeq\\
\simeq H^i\left(X, \mathcal{O}_X(-2,-2)\right)\oplus H^i\left(X, \mathcal{O}_X(-3,-1)\right)\oplus H^i\left(X, \mathcal{O}_X(-1,-3)\right)=0
\end{multline*}
by the~Kodaira-type vanishing theorem, where
$\mathcal{O}_X(a,b)$ is the~restriction to $X$ of the~line bundle  on $\mathbb{P}^3\times\mathbb{P}^3$ of degree $(a,b)$.

Let $Q_1$ and $Q_2$ be the~divisors in $\mathbb{P}^3\times\mathbb{P}^3$ containing $X$ that have degree $(2,0)$ and~$(0,2)$, respectively.
Set $V=Q_1\cap Q_2$. Then it follows from the~exact sequence of sheaves
$$
0\to\mathcal{O}_{V}(-1,-1)\to\mathcal{O}_{V}\to \mathcal{O}_X\to 0
$$
that we have an~exact sequence of cohomology groups
\begin{multline*}
H^1\left(V,\left(\Omega^1_{\mathbb{P}^3\times\mathbb{P}^3}\otimes\mathcal{O}_{\mathbb{P}^3\times\mathbb{P}^3}(-1,-1)\right)|_{V}\right)\to\\
\to H^1\left(X, \Omega^1_{\mathbb{P}^3\times\mathbb{P}^3}|_X\otimes\omega_X\right)\to\\
\to H^2\left(V, \left(\Omega^1_{\mathbb{P}^3\times\mathbb{P}^3}\otimes\mathcal{O}_{\mathbb{P}^3\times\mathbb{P}^3}(-2,-2)\right)|_{V}\right).
\end{multline*}
Now, using the~exact sequences
$$
0\to\mathcal{O}_{Q_1}(0,-2)\to\mathcal{O}_{Q_1}\to\mathcal{O}_{V}\to 0,
$$
and
$$
0\to\mathcal{O}_{\mathbb{P}^3\times \mathbb{P}^3}(-2,0)\to\mathcal{O}_{\mathbb{P}^3\times\mathbb{P}^3}\to\mathcal{O}_{Q_1}\to 0,
$$
we get the~exact sequences
\begin{multline*}
H^i\left(Q_1,\left(\Omega_{\mathbb{P}^3\times\mathbb{P}^3}^1\otimes\mathcal{O}_{\mathbb{P}^3\times\mathbb{P}^3}(-i,-i)\right)|_{Q_1}\right)\to\\
\to H^i\left(V, \left(\Omega_{\mathbb{P}^3\times\mathbb{P}^3}^1\otimes\mathcal{O}_{\mathbb{P}^3\times\mathbb{P}^3}(-i,-i)\right)|_{V}\right)\to\\
\to H^{i+1}\left(Q_1,\left(\Omega_{\mathbb{P}^3\times\mathbb{P}^3}^1\otimes\mathcal{O}_{\mathbb{P}^3\times\mathbb{P}^3}(-i,-i-2)\right)|_{Q_1}\right),
\end{multline*}
\begin{multline*}
H^i\left(\mathbb{P}^3\times\mathbb{P}^3,\Omega_{\mathbb{P}^3\times\mathbb{P}^3}^1\otimes\mathcal{O}_{\mathbb{P}^3\times\mathbb{P}^3}(-i,-i)\right)\to\\
\to H^i\left(Q\times\mathbb{P}^3,\left(\Omega_{\mathbb{P}^3\times\mathbb{P}^3}^1\otimes\mathcal{O}_{\mathbb{P}^3\times\mathbb{P}^3}(-i,-i)\right)|_{Q_1}\right)\to\\
\to H^{i+1}\left(\mathbb{P}^3\times\mathbb{P}^3, \Omega^1_{\mathbb{P}^3\times\mathbb{P}^3}\otimes\mathcal{O}_{\mathbb{P}^3\times\mathbb{P}^3}(-i-2,-i)\right),
\end{multline*}
and
\begin{multline*}
H^{i+1}\left(\mathbb{P}^3\times\mathbb{P}^3,\Omega^1_{\mathbb{P}^3\times\mathbb{P}^3}\otimes\mathcal{O}_{\mathbb{P}^3\times\mathbb{P}^3}(-i,-i-2)\right)\to\\
\to H^{i+1}\left(Q_1,\left(\Omega^1_{\mathbb{P}^3\times\mathbb{P}^3}\otimes\mathcal{O}_{\mathbb{P}^3\times\mathbb{P}^3}(-i,-i-2)\right)|_{Q_1}\right)\to\\
\to H^{i+2}\left(\mathbb{P}^3\times\mathbb{P}^3,\Omega^1_{\mathbb{P}^3\times\mathbb{P}^3}\otimes\mathcal{O}_{\mathbb{P}^3\times\mathbb{P}^3}(-i-2,-i-2)\right)
\end{multline*}
for $i\in\{1,2\}$. Thus, it is enough to show that
\begin{align*}
H^i\left(\mathbb{P}^3\times\mathbb{P}^3,\Omega_{\mathbb{P}^3\times\mathbb{P}^3}^1\otimes\mathcal{O}_{\mathbb{P}^3\times\mathbb{P}^3}(-i,-i)\right)&=0,\\
H^{j+1}\left(\mathbb{P}^3\times\mathbb{P}^3,\Omega_{\mathbb{P}^3\times\mathbb{P}^3}^1\otimes\mathcal{O}_{\mathbb{P}^3\times\mathbb{P}^3}(-j,-j-2)\right)&=0
\end{align*}
for $i=1,2,3,4$ and $j=1,2$, which follows from the~Euler exact sequence
$$
0\to\Omega^1_{\mathbb{P}^3\times\mathbb{P}^3}
\to\mathcal{O}_{\mathbb{P}^3\times\mathbb{P}^3}(-1,0)^{\oplus 4}\oplus
\mathcal{O}_{\mathbb{P}^3\times\mathbb{P}^3}(0,-1)^{\oplus 4}\to
\mathcal{O}_{\mathbb{P}^3\times\mathbb{P}^3}^{\oplus 2}\to 0.
$$
This completes the~proof of the~proposition.
\end{proof}

Now, let  $\mathbf{M}$ be the~irreducible component of the~K-moduli space $\mathrm{M}^\mathrm{Kps}_{3,24}$
whose  points parametrise smooth Fano 3-folds in the~family \textnumero 4.1 and their K-polystable degenerations.
Recall from Section~\ref{section:geometry} that the~GIT moduli space of divisors of degree $(1,1,1,1)$ in $(\mathbb{P}^1)^4$
for the~natural action of the~group $\Gamma=\operatorname{SL}^4_4(\mathbb{C})\rtimes\mathfrak{S}_4$ is isomorphic to
$$
\mathbb{P}^3_{a,b,c,d}\slash \big(W(\mathrm{F}_4)\big/\langle -1\rangle\big)\simeq\mathbb{P}(1,3,4,6),
$$
where $\mathbb{P}^3_{a,b,c,d}$ is identified with the~projectivisation of the~vector space of forms \eqref{form:1},
and the~action of the~finite group $W(\mathrm{F}_4)/\langle -1\rangle\simeq(\mathfrak{A}_4\times\mathfrak{A}_4)\rtimes(\mathbb{Z}/2\mathbb{Z})^2$ together with the~corresponding quotient map $\mathbb{P}^3_{a,b,c,d}\to\mathbb{P}(1,3,4,6)$
are explicitly described in Section~\ref{section:geometry}.
Recall that the~orbit of the~point $(0:0:0:1)\in\mathbb{P}_{a,b,c,d}$ parametrise reducible divisors.

\begin{theorem}
\label{theorem:main}
The following assertions hold:
\begin{itemize}
\item there exists $W(\mathrm{F}_4)/\langle -1\rangle$-equivariant commutative diagram
$$
\xymatrix{
\mathbb{B}\ar@{->}[d]\ar@{->}[rr]&&\mathbf{M}\ar@{->}[d]\\%
\mathbb{P}^3_{a,b,c,d}\ar@{->}[rr]&& \mathbb{P}(1,3,4,6)}
$$
where $\mathbb{B}\to\mathbb{P}^3_{a,b,c,d}$ is the~blow up of the~$W(\mathrm{F}_4)/\langle -1\rangle$-orbit of the~point $(0:0:0:1)$,
both horizontal arrows are the~quotients by the~action of the~group $W(\mathrm{F}_4)/\langle -1\rangle$,
and $\mathbf{M}\to\mathbb{P}(1,3,4,6)$ is a~weighted blow up of a~smooth point with weights $(1,2,3)$,
\item $\mathbf{M}$ is a~connected component of $\mathrm{M}^\mathrm{Kps}_{3,24}$.
\end{itemize}
\end{theorem}

\begin{proof}
In Section~\ref{section:2-2-divisors}, we have constructed a~flat deformation family
$\mathcal{X}\to\mathbb{B}$ of K-polystable 3-folds of the~forms \eqref{equation:form-1} or \eqref{equation:form-2}.
Thus, we can canonically get a~finite morphism
$$
\mathbb{B}\to \mathrm{M}^\mathrm{Kps}_{3,24}.
$$
Since the~rational map $\mathbb{P}^3_{a,b,c,d}\dasharrow \mathrm{M}^\mathrm{Kps}_{3,24}$ and
the blow-up $\mathbb{B}\to \mathbb{P}^3_{a,b,c,d}$ are both equivariant for the~$W(\mathrm{F}_4)/\langle -1\rangle$-action,
the~composition $\mathbb{B}\to \mathrm{M}^\mathrm{Kps}_{3,24}$ is also $W(\mathrm{F}_4)/\langle -1\rangle$-equivariant.
By the~universal property for the~quotient of finite groups,
we get the~morphism
$$
\mathbb{B}\slash \big(W(\mathrm{F}_4)\big/\langle -1\rangle\big)\to \mathrm{M}^\mathrm{Kps}_{3,24}.
$$
This morphism is finite, and it is birational onto its~image.
Moreover, by Proposition~\ref{proposition:deform}, every point in the~image is normal,
so the~morphism is an~isomorphism onto a~connected component $\mathbf{M}$ of $\mathrm{M}^\mathrm{Kps}_{3,24}$ by Zariski's main theorem.
\end{proof}

Let us conclude this section with a~more detailed~structure of the~moduli space $\mathbf{M}$.

\begin{proposition}\label{proposition:toric}
The 3-fold $\mathbf{M}$ is isomorphic to the~weighted blow-up of $\mathbb{P}(1_{H},3_{R},4_{S},6_{T})$
at the~point $(2:2:0:0)$ with weights $(1,2,3)$ at the~regular sequence of parameters
$$
\mathbf{r}=\frac{4R-H^3}{H^3}, \quad
\mathbf{s}=\frac{S}{H^4}, \quad
\mathbf{t}=\frac{T}{H^6},
$$
and $\mathbf{M}$ is isomorphic to the~toric 3-fold $X_{\Sigma}$,  where $\Sigma$ is the~3-dimensional fan whose one-dimensional cones have primitive generators
$$
\mathbf{v}_0=(1,2,3), \quad
\mathbf{v}_1=(1,0,0), \quad
\mathbf{v}_2=(0,1,0), \quad
\mathbf{v}_3=(0,0,1), \quad
\mathbf{v}_4=(-3,-4,-6),
$$
and whose three-dimensional cones are
\begin{align*}
\operatorname{Cone}\{\mathbf{v}_0,\mathbf{v}_1,\mathbf{v}_2\},&\quad\operatorname{Cone}\{\mathbf{v}_0,\mathbf{v}_1,\mathbf{v}_3\},\\
\operatorname{Cone}\{\mathbf{v}_0,\mathbf{v}_2,\mathbf{v}_3\},&\quad \operatorname{Cone}\{\mathbf{v}_4, \mathbf{v}_1,\mathbf{v}_2\},\\
\operatorname{Cone}\{\mathbf{v}_4, \mathbf{v}_1, \mathbf{v}_3\},&\quad\operatorname{Cone}\{\mathbf{v}_4, \mathbf{v}_2, \mathbf{v}_3\}.
\end{align*}
\end{proposition}

\begin{proof}
Let $\mathbb{E}$ be the~exceptional divisor of the~blow-up $\mathbb{B}\to \mathbb{P}^3_{a,b,c,d}$,
and let
$\mathcal{I}$ be the~ideal sheaf on $\mathbb{P}(1_{H},3_{R},4_{S},6_{T})$ with co-support $\{(2:2:0:0)\}$ defined locally by
$$
\mathcal{I}=\left(\mathbf{r}^6, \mathbf{r}^4\mathbf{s},
\mathbf{r}^3\mathbf{t},
\mathbf{r}^2\mathbf{s}^2,
\mathbf{r}\mathbf{s}\mathbf{t},
\mathbf{s}^3, \mathbf{t}^2\right).
$$
Then $X_\Sigma$ is the~blow-up of $\mathbb{P}(1_{H},3_{R},4_{S},6_{T})$ along $\mathcal{I}$.
It follows from Proposition \ref{proposition:explicit} that
$$
\eta^*\big(\mathcal{I}\big)\cdot\mathcal{O}_{\mathbb{B}}=\mathcal{O}_{\mathbb{B}}\big(-12\mathbb{E}\big),
$$
where $\eta$ is the~composition $\mathbb{B}\to \mathbb{P}^3_{a,b,c,d}\to\mathbb{P}(1_{H},3_{R},4_{S},6_{T})$.
This gives $\mathbf{M}\simeq X_\Sigma$, because the~morphism $\eta$ naturally factors through $X_\Sigma$.
\end{proof}

\section{K-semistable limits}
\label{section:stack}

In Section~\ref{section:K-moduli}, we have described all K-polystable Fano 3-folds that are limits of smooth divisors in $(\mathbb{P}^1)^4$ of degree $(1,1,1,1)$.
Now, we prove the~following result:

\begin{theorem}
\label{theorem:Kss}
Let $X$ be a~K-polystable Fano 3-fold admiting a~$\mathbb{Q}$-Gorenstein smoothing~to a~divisor in $(\mathbb{P}^1)^4$ of degree $(1,1,1,1)$.
Then $X$ is isomorphic to an~irreducible complete intersection of three divisors in $\mathbb{P}^3\times\mathbb{P}^3$ of degree $(2,0)$,
$(0,2)$, $(1,1)$.
\end{theorem}

\begin{proof}
Let $\pi\colon\mathcal{X}\to\mathbb{A}^1_t$ be the~K-polystable degeneration of $X$,
and let $X_t$ be the~fiber of the~morphism $\pi$ over the~point $t\in\mathbb{A}^1_t$. Then
$$
X_t\simeq X
$$
for any $t\neq 0$, and $X_0$ is one of K-polystable Fano 3-folds described in Main Theorem.
Since $X_0$ has canonical Gorenstein singularities, so is $\mathcal{X}$.
Therefore, the~3-fold $X_t$ also has canonical Gorenstein singularities for every $t\in\mathbb{A}^1$.

On the~other hand, we know from Main Theorem that $\operatorname{Pic}(X_0)$ contains two Cartier divisors $H_1$ and $H_2$ such that $-K_X\sim H_1+H_2$,
$H_1^3=H_2^3=0$, $H_1^2\cdot H_2=H_1\cdot H_2^2=4$,
both linear system $|H_1|$ and $|H_2|$ are base point free, and give two morphisms $X_0\to \mathbb{P}^3$.
Moreover, arguing as in the~proof of \cite[Lemma 2.17]{Zhuang2020},
we conclude that the~restriction homomorphism $\operatorname{Pic}(\mathcal{X})\to\operatorname{Pic}(X_0)$ is an~isomorphism.
Let $H_1^\mathcal{X}$ and $H_2^\mathcal{X}$ be Cartier divisors on $\mathcal{X}$ such that $H_1^\mathcal{X}\vert_{X_0}=H_1$
and $H_2^\mathcal{X}\vert_{X_0}=H_2$, let $H_{1,t}=H_1^\mathcal{X}\vert_{X_t}$ and $H_{2,t}=H_2^\mathcal{X}\vert_{X_t}$. Then
$$
h^i\big(X_t, H_{1,t}\big)=h^i\big(X_t, H_{2,t}\big)=0
$$
for every $i\geqslant 1$ by the~Kawamata--Viehweg vanishing theorem. Then
$$
h^0\big(X_t, H_{1,t}\big)=h^0\big(X_t, H_{2,t}\big)=4
$$
by the~Riemann--Roch theorem. This implies that $\pi_*(H_i^{\mathcal{X}})$ is locally free sheaf of rank $4$.
Again by the~Kawamata--Viehweg vanishing theorem, the~restriction
$$
\pi_*(H_i^{\mathcal{X}})\to\left(\pi_*(H_i^{\mathcal{X}})\right)\big|_{X_t}
$$
is surjective for any $t\in\mathbb{A}^1$,
so the~evaluation homomorphism $\pi^*\pi_*(H_i^{\mathcal{X}}) \to H_i^{\mathcal{X}}$ is surjective.
Thus, for $i\in\{1,2\}$, the~divisor $H_i^{\mathcal{X}}$ gives a~morphism
$$
\phi_i\colon\mathcal{X}\to \mathbb{P}_{\mathbb{A}^1}\big(\pi_*(H_i^{\mathcal{X}})\big)\cong\mathbb{P}^3\times\mathbb{A}^1,
$$
and its restriction $\phi_{i,t}\colon X_t\to\mathbb{P}^3$ is given by the~complete linear series $|H_{i,t}|$.

For each $i\in\{1,2\}$ and every $t\in\mathbb{A}^1$, let $Q_{i,t}\subset\mathbb{P}^3$ be the~image of the~morphism $\phi_{i,t}$.
Since $(H_{1,t}^3)=(H_{2,t}^3)=0$ and $-K_{X_t}\cdot H_{1,t}^2=-K_{X_t}\cdot H_{2,t}^2=4$,
we see that $Q_{i,t}$ is a~surface, and the~degree of $Q_{i,t}$ is either $2$ or $4$.
If the~degree is $4$, then
$$
-K_{X_t}\cdot\ell=1
$$
for a~general fiber $\ell$ of $\phi_{i,t}\colon X_t\twoheadrightarrow Q_{i,t}$.
But the~Stein factorization of $\phi_{i,t}$ gives a~conic bundle, so $-K_{X_t}\cdot\ell$ is divisible by $2$.
Thus, we see that $Q_{i,t}$ is a~normal quadric surface.

Let $Q_i^{\mathcal{X}}\subset\mathbb{P}^3\times\mathbb{A}^1$ be the~scheme-theoretic image of $\phi_i$.
Since $Q_i^{\mathcal{X}}$ is an~irreducible and reduced Cartier divisor on $\mathbb{P}^3\times\mathbb{A}^1$,
it is Cohen--Macaulay.
Therefore, the~scheme-theoretic fiber of $Q_i^{\mathcal{X}}\to\mathbb{A}^1$ over $t\in\mathbb{A}^1$ is the~surface $Q_{i,t}$,
because this is true for a~general $t\in\mathbb{A}^1$.

Let us consider the~product morphism
$$
\phi=\phi_1\times\phi_2\colon\mathcal{X}\to\mathbb{P}^3\times\mathbb{P}^3\times\mathbb{A}^1
$$
over $\mathbb{A}^1$. Then $\phi$ factors through $\phi\colon\mathcal{X}\to Q_1^{\mathcal{X}}\times_{\mathbb{A}^1} Q_2^{\mathcal{X}}$, and $\phi$ is finite, since
$$
H_1^{\mathcal{X}}+H_2^{\mathcal{X}}\sim_{\mathbb{A}^1}-K_{\mathcal{X}/\mathbb{A}^1}.
$$
Moreover, the~restriction of $\phi$ at $0\in\mathbb{A}^1$ is a~closed embedding,
so $\phi$ is a~closed embedding near $X_0$ (cf.\ \cite[\S 5.1, Exercise 1.25]{Liu2002}).
Then $\phi\colon \mathcal{X}\hookrightarrow Q_1^{\mathcal{X}}\times_{\mathbb{A}^1}Q_2^{\mathcal{X}}$ is a~closed embedding.

Let us consider the~reflexive sheaf
$\mathcal{O}_{Q_1^{\mathcal{X}}\times_{\mathbb{A}^1}Q_2^{\mathcal{X}}}\left(\phi(\mathcal{X})\right)$ on $Q_1^{\mathcal{X}}\times_{\mathbb{A}^1}Q_2^{\mathcal{X}}$.
Note that
$$
\mathcal{O}_{Q_1^{\mathcal{X}}\times_{\mathbb{A}^1}Q_2^{\mathcal{X}}}\left(\phi(\mathcal{X})\right)\big|_{Q_{1,0}\times Q_{2,0}}\simeq\mathcal{O}_{Q_{1,0}\times Q_{2,0}}(X_0)
$$
by \cite[Lemma 12.1.8]{KollarMori}.
Since $\mathcal{O}_{Q_{1,0}\times Q_{2,0}}(X_0)$ is locally free, so is
$\mathcal{O}_{Q_1^{\mathcal{X}}\times_{\mathbb{A}^1}Q_2^{\mathcal{X}}}
\left(\phi(\mathcal{X})\right)$ around the~fiber over $0\in\mathbb{A}^1$.
Therefore, for any $t\in\mathbb{A}^1$, the~image of the~embedding
$$
X_t\hookrightarrow Q_{1,t}\times Q_{2,t}
$$
is a~Cartier divisor. Thus, since $-K_{X_t}\sim H_{1,t}+H_{2,t}$,
we have $X_t\in\left|\mathcal{O}_{Q_{1,t}\times Q_{2,t}}(1,1)\right|$.
\end{proof}

Now, we set
$$
\mathbb{H}=
\mathbb{P}\Big(H^0\left(\mathbb{P}^3\times\mathbb{P}^3,\mathcal{O}(2,0)\right)\Big)\times
\mathbb{P}\Big(H^0\left(\mathbb{P}^3\times\mathbb{P}^3,\mathcal{O}(0,2)\right)\Big)\times
\mathbb{P}\Big(H^0\left(\mathbb{P}^3\times\mathbb{P}^3,\mathcal{O}(1,1)\right)\Big),
$$
Then $\mathbb{H}\simeq\mathbb{P}^{9}\times\mathbb{P}^{9}\times\mathbb{P}^{15}$ is the~space parametrising complete intersections in $\mathbb{P}^3\times\mathbb{P}^3$ of three divisors of degree $(2,0)$, $(0,2)$, $(1,1)$.
Let $\mathbb{H}^{\operatorname{{Kss}}}\subset\mathbb{H}$ be the~locus whose closed points parametrise K-semistable complete intersections.
Then
$$
\mathrm{codim}\big(\mathbb{H}\setminus\mathbb{H}^{\operatorname{{Kss}}}\big)\geqslant 2.
$$
Let $\lambda_{\operatorname{CM}}$ be the~CM line bundle on $\mathbb{H}^{\operatorname{Kss}}$ of the~corresponding universal family of complete intersections,
and let ${\mathcal{X}}^{\mathbb{H}}\to \mathbb{H}$ be the~universal family of complete intersections.
Then
\begin{align*}
\left((-K_{\mathcal{X}^{\mathbb{H}}/\mathbb{H}})^4\cdot
\mathcal{O}_{\mathbb{H}}(1,0,0)|_{\mathcal{X}^{\mathbb{H}}}^{8}
\cdot\mathcal{O}_{\mathbb{H}}(0,1,0)|_{\mathcal{X}^{\mathbb{H}}}^{9}
\cdot\mathcal{O}_{\mathbb{H}}(0,0,1)|_{\mathcal{X}^{\mathbb{H}}}^{15}\right)&=-76,\\
\left((-K_{\mathcal{X}^{\mathbb{H}}/\mathbb{H}})^4\cdot
\mathcal{O}_{\mathbb{H}}(1,0,0)|_{\mathcal{X}^{\mathbb{H}}}^{9}
\cdot\mathcal{O}_{\mathbb{H}}(0,1,0)|_{\mathcal{X}^{\mathbb{H}}}^{8}
\cdot\mathcal{O}_{\mathbb{H}}(0,0,1)|_{\mathcal{X}^{\mathbb{H}}}^{15}\right)&=-76,\\
\left((-K_{\mathcal{X}^{\mathbb{H}}/\mathbb{H}})^4\cdot
\mathcal{O}_{\mathbb{H}}(1,0,0)|_{\mathcal{X}^{\mathbb{H}}}^{9}
\cdot\mathcal{O}_{\mathbb{H}}(0,1,0)|_{\mathcal{X}^{\mathbb{H}}}^{9}
\cdot\mathcal{O}_{\mathbb{H}}(0,0,1)|_{\mathcal{X}^{\mathbb{H}}}^{14}\right)&=-72.
\end{align*}
This implies that $\lambda_{\operatorname{CM}}$ is the~restriction of the~line bundle $\mathcal{O}_{\mathbb{H}}(76,76,72)$.

\appendix

\section{Four lines in $\mathbb{P}^3$}
\label{section:4-lines}

Recall from Section~\ref{section:geometry} that $\mathbb{P}^{\operatorname{Sing}}$ is the~divisor in $\mathbb{P}^3_{a,b,c,d}$
whose closed points parametrise singular divisors in $\mathbb{P}^1_{x_1,y_1}\times\mathbb{P}^1_{x_2,y_2}\times\mathbb{P}^1_{x_3,y_3}\times\mathbb{P}^1_{x_4,y_4}$  that are given by
\begin{multline}
\label{equation:1-1-1-1}
\frac{a+d}{2}\big(x_1x_2x_3x_4+y_1y_2y_3y_4\big)
+\frac{a-d}{2}\big(x_1x_2y_3y_4+y_1y_2x_3x_4\big)+\\
+\frac{b+c}{2}\big(x_1y_2x_3y_4+y_1x_2y_3x_4\big)
+\frac{b-c}{2}\big(x_1y_2y_3x_4+y_1x_2x_3y_4\big)=0.
\end{multline}
Recall from Section~\ref{section:2-2-divisors} that $\mathbb{B}$ is the~blow up of $\mathbb{P}^3_{a,b,c,d}$ along $12$ points parametrising reducible divisors,
and $\mathbb{E}=\mathbb{P}^2_{a,b,c}$ is one of the~exceptional divisors of this blow up whose closed points parametrise divisors in $\mathbb{P}(1_{s_1},1_{t_1},2_{w_1})\times\mathbb{P}(1_{s_2},1_{t_2},2_{w_2})$ that are given by
\begin{equation}
\label{equation:2-2}
w_1w_2=as_1t_1s_2t_2 +\frac{b+c}{4}\big(s_1^2s_2^2+t_1^2t_2^2\big)+\frac{b-c}{4}\big(s_1^2t^2_2+t^2_1s^2_2\big).
\end{equation}

Let $\widetilde{\mathbb{P}^{\operatorname{Sing}}}$ be the~strict transform on $\mathbb{B}$ of the~divisor $\mathbb{P}^{\operatorname{Sing}}$.
If $X$ is a singular~Fano~\mbox{3-fold} which is parametrised by a point of $\widetilde{\mathbb{P}^{\operatorname{Sing}}}$, it can be described in a~simple geometric~way.
For instance, if $X$ is a~divisor \eqref{equation:1-1-1-1} with two singular points,
then we have the~following commutative diagram:
\begin{equation}
\label{equation:4-lines}
\xymatrix{
&Y\ar@{->}[dl]_{\pi}\ar@{->}[dr]^{\eta}\\%
\mathbb{P}^3\ar@{-->}[rr]_{\chi} && X}
\end{equation}
where $\pi$ is a~blow up of four disjoint lines $L_1$, $L_2$, $L_3$, $L_4$ that are not contained in a~quadric surface,
$\chi$ is induced by the~composition of linear projections $\mathbb{P}^3\dasharrow\mathbb{P}^1$ from these lines,
and $\eta$ is a~small contraction of the~strict transforms of the~two lines in $\mathbb{P}^3$ that intersect the~lines $L_1$, $L_2$, $L_3$, $L_4$.
If $X$ is a~divisor \eqref{equation:1-1-1-1} singular along a~curve,
we still has \eqref{equation:4-lines}, but in this case the~lines $L_1$, $L_2$, $L_3$, $L_4$ are contained in a~quadric surface,
and $\eta$ contracts the~strict transform of this surface.

\subsection{Disjoint lines}
\label{subsection:4-lines-disjoint}
In this section, we will show the~existence of the~diagram \eqref{equation:4-lines}
in the~case when $X$ is a~divisor \eqref{equation:1-1-1-1}
whose singular locus is either two points or a~curve.

Let $G\simeq(\mathbb{Z}/2\mathbb{Z})^3$ be the~subgroup in $\mathrm{Aut}(\mathbb{P}^3_{z_0,z_1,z_2,z_3})$ generated by the~involutions:
\begin{align*}
\tau_1\colon (z_0:z_1:z_2:z_3)&\mapsto(z_1:z_0:z_3:z_2);\\
\tau_2\colon (z_0:z_1:z_2:z_3)&\mapsto(z_0:-z_1:z_2:-z_3);\\
\tau_3\colon (z_0:z_1:z_2:z_3)&\mapsto(z_2:z_3:z_0:z_1).
\end{align*}
Fix $(c_0:c_1:c_2:c_3)\in\mathbb{P}^3$ such that $c_0^2\ne c_2^2$,
and let $L_1$ be the~line in $\mathbb{P}^3$ that passes through the~points $(c_0:c_1:c_2:c_3)$ and $(c_2:c_3:c_0:c_1)=\tau_3(c_0:c_1:c_2:c_3)$.
Note that
$$
(c_2:c_3:c_0:c_1)\ne (c_0:c_1:c_2:c_3).
$$
Similarly, we define the~lines $L_2$, $L_3$, $L_4$ as follows: $L_2=\tau_1(L_1)$, $L_3=\tau_2(L_2)$, $L_4=\tau_2(L_1)$.
Then $L_1=\{f_1=g_1=0\}$, $L_2=\{f_2=g_2=0\}$, $L_3=\{f_3=g_3=0\}$, $L_4=\{f_4=g_4=0\}$, where
\begin{align*}
f_1&=(c_2c_3-c_0c_1)z_0+(c_0^2-c_2^2)z_1+(c_1c_2-c_0c_3)z_2,\\
g_1&=(c_1c_2-c_0c_3)z_0+(c_2c_3-c_0c_1)z_2+(c_0^2-c_2^2)z_3,\\
f_2&=(c_0^2-c_2^2)z_0+(c_2c_3-c_0c_1)z_1+(c_1c_2-c_0c_3)z_3,\\
g_2&=(c_1c_2-c_0c_3)z_1+(c_0^2-c_2^2)z_2+(c_2c_3-c_0c_1)z_3,\\
f_3&=(c_0^2-c_2^2)z_0+(c_0c_1-c_2c_3)z_1+(c_0c_3-c_1c_2)z_3,\\
g_3&=(c_0c_3-c_1c_2)z_1+(c_0^2-c_2^2)z_2+(c_0c_1-c_2c_3)z_3,\\
f_4&=(c_2c_3-c_0c_1)z_0-(c_0^2-c_2^2)z_1+(c_1c_2-c_0c_3)z_2,\\
g_4&=(c_1c_2-c_0c_3)z_0+(c_2c_3-c_0c_1)z_2-(c_0^2-c_2^2)z_3.
\end{align*}
Set
\begin{align*}
D_{12}&=\big((c_0+c_1)^2-(c_2+c_3)^2\big)\big((c_0-c_1)^2-(c_2-c_3)^2\big), \\
D_{13}&=\big((c_0-c_2)^2+(c_1-c_3)^2\big)\big((c_0+c_2)^2+(c_1+c_3)^2\big),\\
D_{14}&=\big(c_0^2-c_2^2\big)\big(c_1^2-c_3^2\big).
\end{align*}
Then the~lines $L_1$, $L_2$, $L_3$, $L_4$ are disjoint if and only if
\begin{equation}
\label{2}
D_{12}D_{13}D_{14}\neq 0.
\end{equation}
Similarly, the~lines $L_1$, $L_2$, $L_3$, $L_4$ are
distinct but each of them intersects with exactly
two of the~others if and only if
\begin{equation}\label{6}
\begin{split}
D_{12}=0, D_{13}=0, & D_{14}\neq 0, \text{ or }\\
D_{12}=0, D_{14}=0, & D_{13}\neq 0, \text{ or}\\
D_{13}=0, D_{14}=0, & D_{12}\neq 0.
\end{split}
\end{equation}
From now on, we assume that the~lines $L_1$, $L_2$, $L_3$, $L_4$ are disjoint.

Let $\chi\colon\mathbb{P}^3\dasharrow(\mathbb{P}^1)^4$ be the~map given by
\begin{equation}
\label{chi}
(z_0:z_1:z_2:z_3)\mapsto\big((f_1:g_1),(f_2:g_2),(f_3:g_3),(f_4:g_4)\big),
\end{equation}
and let $X$ be the~closure of its image in $(\mathbb{P}^1)^4$.
Then $\chi$ is $G$-equivariant,
and $X$ is a~divisor in $(\mathbb{P}^1)^4$ of degree $(1,1,1,1)$ which is given by the~following equation:
\begin{multline*}
\quad \quad \quad \quad 4(c_0c_3-c_1c_2)^2(x_1x_2x_3x_4+y_1y_2y_3y_4)-\\
-(c_0^2-c_1^2-c_2^2+c_3^2)^2(x_1x_2y_3y_4+x_3x_4y_1y_2)+\quad\quad  \\
\quad \quad +(c_0^2+c_1^2-c_2^2-c_3^2)^2(x_1x_3y_2y_4+x_2x_4y_1y_3)-\\
-4(c_0c_1-c_2c_3)^2(x_1x_4y_2y_3+x_2x_3y_1y_4)=0.\quad \quad \quad \quad
\end{multline*}
This is equation \eqref{equation:1-1-1-1} with
\begin{equation}
\label{abcd}
\begin{cases}
a=\big((c_0+c_1)^2-(c_2+c_3)^2\big)\big((c_2-c_3)^2-(c_0-c_1)^2\big), \\
b=\big((c_0+c_1)^2-(c_2+c_3)^2\big)\big((c_0-c_1)^2-(c_2-c_3)^2\big), \\
c=(c_0^2+c_1^2-c_2^2-c_3^2)^2+4(c_0c_1-c_2c_3)^2, \\
d=4(c_0c_3-c_1c_2)^2+(c_0^2-c_1^2-c_2^2+c_3^2),
\end{cases}
\end{equation}
so that we have $a+b=0$.
Moreover, we have $G$-equivariant commutative diagram~\eqref{equation:4-lines},
where $\eta$ is the~birational morphism contracting strict transforms of all lines in $\mathbb{P}^3$
that intersect the~lines $L_1$, $L_2$, $L_3$, $L_4$.
In particular, if $L_1$, $L_2$, $L_3$, $L_4$ are contained in a~quadric, then $\eta$ contracts a~divisor, and $X$ is singular along a~curve.
This does not happen if
\begin{equation}
\label{3}
\big(c_0^2+c_1^2-c_2^2-c_3^2\big)\big(c_0^2-c_1^2-c_2^2+c_3^2\big)\big(c_0c_1-c_2c_3\big)\big(c_0c_3-c_1c_2\big)\ne 0.
\end{equation}
Moreover, if $L_1$, $L_2$, $L_3$, $L_4$ are not contained in a~quadric, then
there exist exactly two distinct lines in $\mathbb{P}^3$ that intersect all lines $L_1$, $L_2$, $L_3$, $L_4$.

We have constructed \eqref{equation:4-lines} such that $X$ is a divisor \eqref{equation:1-1-1-1} that either has two nodes or a curve of singularities.
However, it is not clear that we can obtain all divisors \eqref{equation:1-1-1-1} that have these properties. This follows from the~following lemmas.

\begin{lemma}
\label{lemma:surjective_separate}
Let $\rho\colon\mathbb{P}^3_{z_0,z_1,z_2,z_3}\dashrightarrow\mathbb{P}^3_{x,y,z,t}$ be the~rational map given by
$$
(z_0:z_1:z_2:z_3)\mapsto\big(2(z_0z_3-z_1z_2):(z_0^2-z_1^2-z_2^2+z_3^2):2(z_0z_1-z_2z_3):(z_0^2+z_1^2-z_2^2-z_3^2)\big),
$$
let $\sigma\colon\mathbb{P}^3_{x,y,z,t}\to\mathbb{P}^3_{a,b,c,d}$ be the~morphism given by
$$
(x:y:z:t)\mapsto(x^2-y^2:-z^2+t^2:z^2+t^2:x^2+y^2).
$$
Take any point $(a:b:c:d)\in \mathbb{P}^3\setminus(1:-1:1:1)$ such that $a+b=0$.
Then there exists a point $(c_0:c_1:c_2:c_3)\in\mathbb{P}^3$
outside of the~indeterminacy locus of $\sigma\circ\rho$ such that
$$
\sigma\circ\rho(c_0:c_1:c_2:c_3)=(a:b:c:d)
$$
and $c_0^2\neq c_2^2$.
\end{lemma}

\begin{proof}
Observe that the~image of $\rho$ is the~quadric surface $Q=\{x^2-y^2-z^2+t^2=0\}$,
the~image $\sigma(Q)$ is the~plane $\{a+b=0\}$, and $\rho$ fits the~following commutative diagram:
$$
\xymatrix{
&W\ar@{->}[dl]_{\varpi}\ar@{->}[dr]^{\varphi}\\%
\mathbb{P}^3_{z_0,z_1,z_2,z_3}\ar@{-->}[rr] && \mathbb{P}^1_{s_1,t_1}\times\mathbb{P}^1_{s_2,t_2}\simeq Q}
$$
where $\varpi$ is the~blow up of the~lines $\{z_0-z_2=z_3-z_1=0\}$ and $\{z_0+z_2=z_1+z_3=0\}$,
and $\varphi$ is a~$\mathbb{P}^1$-bundle.
Moreover, the~map $\mathbb{P}^3_{z_0,z_1,z_2,z_3}\dashrightarrow \mathbb{P}^1_{s_i,t_i}$ is given by
$$
(z_0:z_1:z_2:z_3)\mapsto \big((z_0-z_2:z_3-z_1),(z_1+z_3:z_0+z_2)\big),
$$
and the~isomorphism $\mathbb{P}^1_{s_1,t_1}\times\mathbb{P}^1_{s_2,t_2}\simeq Q$ is given by
$$
\big((s_1:t_1),(s_2:t_2)\big)\mapsto\big(s_1s_2+t_1t_2: s_1t_2+s_2t_1: s_1s_2-t_1t_2: s_1t_2-s_2t_1\big).
$$
Since $(a:b:c:d)\neq(1:-1:1:1)$, there exists a point $((s_1:t_1),(s_2:t_2))$ with $s_1t_2\neq 0$ which is mapped by $\sigma$ to $(a:b:c:d)$ .
Since $s_1t_2\neq 0$,
the image of the~fiber of $\varphi$ over the~point $((s_1:t_1),(s_2:t_2))$ is not contained by any plane of the~form $\{z_0-\lambda z_2=0\}$,
which implies the~required assertion.
\end{proof}

\begin{corollary}
\label{corollary:4disjoint}
Let $(a:b:c:d)$ be a point in $\mathbb{P}^3$ such that $a+b=0$, and \eqref{equation:1-1-1-1} defines a~divisor $X\subset (\mathbb{P}^1)^4$
such that $X$ has either two singular points or a~curve of singularities.
Then there is $(c_0:c_1:c_2:c_3)\in\mathbb{P}^3$ satisfying \eqref{2} and \eqref{abcd}.
\end{corollary}

\begin{proof}
By Corollary~\ref{corollary:singular}, $a(c^2-d^2)\neq 0$,
so the~assertion follows from Lemma \ref{lemma:surjective_separate}.
\end{proof}

\subsection{Intersecting lines}
\label{subsection:4-lines-intersecting}
What we did in Appendix~\ref{subsection:4-lines-disjoint} works also in some cases when the~lines $L_1$, $L_2$, $L_3$, $L_4$ are not disjoint.
Namely, let $X\subset(\mathbb{P}^1)^4$ be the~divisor \eqref{equation:1-1-1-1} for
$$
(a:b:c:d)=(-u:u:1:1)
$$
with some $u\in\mathbb{C}\setminus\{\pm 1\}$.
If $u\ne 0$, then $X$ has $4$ ordinary double points by Theorem~\ref{theorem:sing}.
Similarly, if $u=0$, then $X$ has $6$ ordinary double points.
Moreover, by Lemma \ref{lemma:surjective_separate}, there exists $(c_0:c_1:c_2:c_3)\in\mathbb{P}^3$ such that $c_0^2\neq c_2^2$ and
$$
\sigma\circ\rho(c_0:c_1:c_2:c_3)=(-u:u:1:1).
$$
Replacing $(c_2,c_3)$ with $(-c_2,-c_3)$ if necessary, we may assume that $c_1=c_3$. Then
$$
(u:1)=\left(4c_1^2-(c_0+c_2)^2:4c_1^2+(c_0+c_2)^2\right).
$$
Note that $L_1\cap L_4=(-1:0:1:0)$ and $L_2\cap L_3=(0:1:0:-1)$. If $u\ne 0$, then
$$
(L_1\cup L_4)\cap (L_2\cup L_3)=\varnothing.
$$
Similarly, if $u=0$, then $L_{1}\cap L_{2}=(1:\pm 1:1:\pm 1)$ and $L_{3}\cap L_{4}=(-1:\pm 1:-1:\pm 1)$.

The image of the~rational map $\chi\colon\mathbb{P}^3\dasharrow(\mathbb{P}^1)^4$ given by \eqref{chi} defines our divisor~$X$.
Moreover, if $u\neq 0$, we have the~following $G$-equivariant commutative diagram:
\begin{equation}
\label{equation:4-lines-4-nodes}
\xymatrix{
U\ar@{->}[d]_{\pi_1}&&Y\ar@{->}[ll]_{\pi_2}\ar@{->}[d]^{\eta}\\%
\mathbb{P}^3\ar@{-->}[rr]_{\chi} && X}
\end{equation}
where $\pi_1$ blows up the~points $L_1\cap L_4$ and $L_2\cap L_3$,
$\pi_2$ blows up the~strict transforms of the~lines $L_1$, $L_2$, $L_3$, $L_4$,
and $\eta$ is the~anticanonical morphism.
Note that $\eta$ contracts the~strict transforms of the~two lines in $\mathbb{P}^3$
that intersect  $L_1$, $L_2$, $L_3$, $L_4$ --- these are the~line that passes through $L_1\cap L_4$ and $L_2\cap L_3$,
and the~intersection line of the~planes containing the~curves $L_1+L_4$ and $L_2+L_3$, respectively.
The~morphism $\eta$ also contracts the~strict transforms of the~two lines in the~$\pi_1$-exceptional divisors
that pass through the~intersection points with strict transforms of $L_1$, $L_2$, $L_3$, $L_4$.

The~diagram \eqref{equation:4-lines-4-nodes} exists for $u=0$.
In this case, $\pi_1$ is the~blow up of the~intersection points $L_1\cap L_4$, $L_2\cap L_3$, $L_{1}\cap L_{2}$, $L_{3}\cap L_{4}$,
and $\eta$ contracts the~strict transforms of the~two lines in $\mathbb{P}^3$
that intersect  $L_1$, $L_2$, $L_3$, $L_4$ --- the~line that passes through $L_1\cap L_4$ and $L_2\cap L_3$,
and the~line that passes through $L_{1}\cap L_{2}$ and $L_{3}\cap L_{4}$.
In addition to these two curves, the~morphism $\eta$ also contracts the~strict transforms of the~four lines in the~$\pi_1$-exceptional divisors
that pass through the~intersection points with strict transforms of   $L_1$, $L_2$, $L_3$, $L_4$.

\subsection{Merging lines}
\label{subsection:2-lines}

The construction described in Appendix~\ref{subsection:4-lines-disjoint} degenerate
in the~case when two pair of lines among $L_1$, $L_2$, $L_3$, $L_4$ merge together.
To be precise, let us fix two disjoint lines $L_1$ and $L_2$ in $\mathbb{P}^3$,
and let $\pi_1\colon U\to\mathbb{P}^3$ be the~blow up of the~lines $L_1$ and $L_2$.
Then we have the~following commutative diagram
$$
\xymatrix{
&U\ar@{->}[dl]_{\pi_1}\ar@{->}[dr]^{\nu}&\\%
\mathbb{P}^3\ar@{-->}[rr]_{\rho}&&\mathbb{P}^1\times\mathbb{P}^1}
$$
where $\rho$ is given by the~product of projections from $L_1$~and~$L_2$, and $\nu$ is a~$\mathbb{P}^1$-bundle.
Let~$E_1$ and $E_2$ be the~$\pi_1$-exceptional surfaces such that $\pi_1(E_1)=L_1$ and $\pi_1(E_2)=L_2$,
let~$L_3\subset E_1$ and $L_4\subset E_2$ be (possibly reducible) curves such that
$\nu(L_2)$ and $\nu(L_3)$ are curves of degree $(1,1)$, and one of the~following cases holds:
\begin{enumerate}
\item[(1)] $L_3$ and $L_4$ are smooth and the~intersection $\nu(L_3)\cap\nu(L_4)$ consists of two points;
\item[(2)] $L_3$ and $L_4$ are singular and the~intersection $\nu(L_3)\cap\nu(L_4)$ consists of two points;
\item[(3)] $L_3$ and $L_4$ are smooth and $\nu(L_3)=\nu(L_4)$.
\end{enumerate}
Let $\pi_2\colon Y\to U$ be the~blow up of the~curves $L_3$ and $L_4$,
and let $\widetilde{E}_1$ and $\widetilde{E}_2$ be the~strict transforms on $Y$ of the~surfaces $E_1$ and $E_2$, respectively.
Then $Y$ is a~weak Fano 3-fold, and the~linear system $|-K_Y|$ gives a~birational morphism $\nu\colon Y\to X$
such that $X$ is a~divisor in $\mathbb{P}(1,1,2)\times\mathbb{P}(1,1,2)$ of degree $(2,2)$ that can be given by \eqref{equation:2-2},
and we have the~following commutative diagram:
$$
\xymatrix{
&U\ar@{->}[dl]_{\pi_1}\ar@{->}[dr]^{\nu}&&Y\ar@{->}[dr]^{\eta}\ar@{->}[ll]_{\pi_2}&\\%
\mathbb{P}^3\ar@{-->}[rr]_{\chi} &&\mathbb{P}^1\times\mathbb{P}^1 &&X\ar@{-->}[ll]^{\rho}}
$$
where the~map $\rho$ is given by the~natural projection $\mathbb{P}(1,1,2)\times\mathbb{P}(1,1,2)\dasharrow\mathbb{P}^1\times\mathbb{P}^1$.

The morphism $\nu$ contracts  $\widetilde{E}_1$ and $\widetilde{E}_2$ such that
$\eta(\widetilde{E}_1)$ and $\eta(\widetilde{E}_2)$ are curves in $\mathrm{Sing}(X)$,
and $\eta$ also contracts strict transforms of the~fibers of the~$\mathbb{P}^1$-bundle $\nu$ that intersect both curves $L_3$ and $L_4$.
In particular, if the~intersection $\nu(L_3)\cap\nu(L_4)$ consists of two points, there are exactly two such fibers --- the~fibers over the~points $\nu(L_3)\cap\nu(L_4)$.
Similarly, if
$$
\nu(L_3)=\nu(L_4),
$$
then such fibers span a~surface in $Y$, whose image in $\mathbb{P}^3$ is a~smooth quadric surface,
so~the~morphism $\eta$~contracts this surface to another curve of singularities of the~3-fold $X$.
In both cases, the~3-fold $X$ is one of the~divisors \eqref{equation:2-2} parametrised by points in $\widetilde{\mathbb{P}^{\operatorname{Sing}}}\cap\mathbb{E}$.
Moreover, we can obtain all such divisors using this construction.

\end{document}